\documentclass[10pt]{amsart}

\usepackage{vmargin}

\setpapersize{A4}
\setmargins{2.75cm}       % margen izquierdo
{1.5cm}                        % margen superior
{15.5cm}                      % anchura del texto
{23.42cm}                    % altura del texto
{10pt}                           % altura de los encabezados
{1cm}                           % espacio entre el texto y los encabezados
{0pt}                             % altura del pie de página
{2cm}

\usepackage{amscd, amssymb, amsfonts, amsthm, amsmath}
\usepackage[all]{xy}
\usepackage{color}
\usepackage{mathdots}
\usepackage{hyperref}

\usepackage{graphicx}

\begin{document}
\newcommand{\mono}[1]{%
\gdef\puA{#1}}
\newcommand{\puA}{}
\newcommand{\faculty}[1]{%
\gdef\puC{#1}}
\newcommand{\puC}{}
\newcommand{\facultad}[1]{%
\gdef\puD{#1}}
\newcommand{\puD}{}
\newtheorem{teo}{Theorem}[section]
\newtheorem{prop}[teo]{Proposition}
\newtheorem{lema}[teo] {Lemma}
\newtheorem{ej}[teo]{Example}
\newtheorem{obs}[teo]{Remark}
\newtheorem{defi}[teo]{Definition}
\newtheorem{coro}[teo]{Corollary}
\newtheorem{nota}[teo]{Notation}

\def\F{\mathbb{F}}
\def\C{\mathbb{C}}
\def\R{\mathbb{R}}
\def\Z{\mathbb{Z}}
\def\Q{\mathbb{Q}}
\def\N{\mathbb{N}}
\def\P{\mathbb{P}}
\def\S{\mathbb{S}}
\def\D{\mathbb{D}}
\def\H{\mathbb{H}}
\def\T{\mathbb{T}}
\def\d{\mathfrak{d}}
\def\AC{\mathcal{AC}}
\def\ACI{\mathcal{PAC}_{+}(\S^{1})}
\def\PC{\mathcal{PC}_{+}(\S^{1})}
\def\IET{\textnormal{IET}}
\def\AIET{\textnormal{AIET}}
\def\PL{\textnormal{PL}}

\newenvironment{note}{\noindent Notation: \rm}

\title{ Continuous homomorphisms on piecewise absolutely continuous maps of $\S^{1}$ }

\author[Barrios]{Marcos Barrios}
\address{Universidad de La Rep\'ublica,
Uruguay} \email{marcosb@fing.edu.uy}

\maketitle

\begin{abstract}{
Let $\IET(\S^{1})$ be the group of interval exchange transformation of $\S^{1}$ and \(\AC_{+}(\S^{1})\) be the group of absolutely continuous preserving orientation bijection with inverse absolutely continuous.
We denote by \(\ACI\) the group generated by \(\IET(\S^{1}\)) and \(\AC_{+}(\S^{1})\).
Given a suitable distance on $\ACI$, we classify all continuous homomorphisms $\rho:\R \to \ACI$. More precisely, \(\rho\) is conjugated to  a continuous homomorphism \(\hat{\rho}:\R \to \AC_{+}(\uplus_{i}(\S^{1})_{i})\).}

\end{abstract} 

\section{Introduction}
An interval exchange transformation of $[0,1)$ is a bijection $f:[0,1) \to [0,1)$ defined by a finite partition of the unit interval into half-open intervals (i.e. \([b_{i},b_{i+1})\)) and rearrangement of these intervals.
More precisely, there exists a permutation $\sigma:\{1,...,n\} \to \{1,...,n\}$ and $b_{i} \in [0,1)$, with \(b_{1} = 0,\, b_{n+1} = 1\) and $b_{i} < b_{j}$ if $i < j$, such that $f(x) = x - b_{j} + \displaystyle\sum_{\sigma(i) < \sigma(j)} (b_{i+1}-b_{i})$ if $x \in [b_{j},b_{j+1})$. 
Let \(\IET\) be the group of all interval exchange transformations of \([0,1)\).

The dynamic of an interval exchange transformation has been studied since the 70's. An important result is that almost every interval exchange transformation is uniquely ergodic, proved independently by Masur \cite{Masur} and Veech \cite{Veech}.

Later \(\IET\) has been studied as a group. For example, in \cite{Novak2},  Novak proved that \(\IET\) does not contain distortion elements.

An important problem about \(\IET\) is the following question formulated by Katok. Does \(\IET\) contain a non abelian free group?

In \cite{Dahmani1}, Dahamani, Fujiwara and Guirardel proved that a group generated by a generic pair of elements of \(\IET\) is not a non abelian free group. Also they proved that a connected Lie group can be embedded in \(\IET\) if and only if it is abelian.
In \cite{Dahmani2} these authors worked in solvable groups of \(\IET\), and proved that some groups cannot be embedded in \(\IET\).

The group \(\IET\) has canonical inclusions into other groups, like interval exchange transformations with flips (FIET), affine interval exchange transformations (AIET) and permutation of \(\S^{1}\) that are continuous outside a finite subset ($\widehat{\text{PC}^{\bowtie}}(\S^{1})$).

These groups have been studied in several works.

Unlike the case of \(\IET\), for FIET that reverse orientation in at least one interval Nogueira proved that almost all have periodic points, in particular they are non-ergodic \cite{Nogeuira}.

The group FIET is simple, this is an unpublished result of Arnoux \cite{Arnoux}.
%and it has recovered by Lacourte \textcolor{red}{[]}.  
Later, Guelman and Liousse calculated an uniform bound for the commutator length \cite{Guelman1}.

In \cite{Boulanger}, Boulanger, Fougeron and Ghazouani, study the dynamics of a 1-parameter family in \(\AIET\) in which the dynamics is trivial in a full measure set but not on a Cantor set.

Some question are simpler in \(\AIET\) than in \(\IET\), for example the Katok question has an affirmative answer in \(\AIET\).
More precisely, via a ping-pong argument, it can be built an explicit example of a non abelian free group in \(\AIET\) \cite{Navas}.

Let ($\mathcal{G}_{fin}(\S^{1})$) be the group of permutation of \(\S^{1}\) with finite support. It is a normal subgroup in $\widehat{\text{PC}^{\bowtie}}(\S^{1})$. 

In \cite{Cornulier} Cornulier studied when a subgroup of $\widehat{\text{PC}^{\bowtie}}(\S^{1})/\mathcal{G}_{fin}(\S^{1})$ can be lifted to $\widehat{\text{PC}^{\bowtie}}(\S^{1})$.
Examples of groups that can be lifted are the subgroups of \(\text{PC}_{+}(\S^{1})\), the group of piecewise continuous bijection of \(\S^{1}\), that are piecewise orientation-preserving. 
In this article we work with a subgroup of \(\text{PC}_{+}(\S^{1})\).

\begin{defi} \label{interval def of ACI}
A bijection $f:\S^{1} \to \S^{1}$ is called a right piecewise absolutely continuous if

\begin{enumerate}
\item is right-continuous,
\item the set of discontinuity is finite,
\item if $f$ is continuous in $[a,b)$ then $f$ is absolutely continuous in $[a,b)$.
\item if $f$ is continuous in $[a,b)$ then $[f(a),f(b)) = f([a,b))$.
\end{enumerate}

We denote by $\ACI$ the group of all bijections \(f:\S^{1} \to \S^{1}\) such that \(f\) and \(f^{-1}\) are right piecewise absolutely continuous.

\end{defi}

The definition of intervals in \(\S^{1}\) is introduced in Notation \ref{notation intervals}.
For precise definition of right-continuous and absolutely continuous see Definition \ref{def rlim rcont abs cont}.

If we exclude the item \(1\) of this definition, we obtain another group denoted by \(\widehat{\mathcal{PAC}}_{+}(\S^{1})\), (following the notation in \cite{Cornulier}).
This group contains the normal subgroup of finitely supported permutation in \(\S^{1}\) denoted by (\(\mathcal{G}_{fin}(\S^{1})\)).

These three groups are related by \(\ACI \simeq (\widehat{\mathcal{PAC}}_{+}(\S^{1})/\mathcal{G}_{fin}(\S^{1}))\). In other words, the condition 1 implies the only function with finite support in \(\ACI\) is the identity function.

In the case that $f$ is continuous, property 4 is equivalent to orientation preserving. Also, property 4 of this definition implies that $f$ preserve orientation on intervals of continuity.
In particular if $f$ is continuous in $[a,b]$ then $f^{-1}$ is continuous in $[f(a),f(b)]$.

The group \(\ACI\) verifies \(\ACI = \langle \IET(\S^{1}), \mathcal{AC}_{+}(\S^{1})\rangle\), in a similar way that AIET verifies  \(\AIET = \langle \IET,\PL_{+}([0,1))\rangle\).

We focus on the work of Novak in \cite{Novak1}. The main result of this article is a classification theorem: given a natural unrestricted topology in \(\IET\), any continuous homomorphism $\tau: \R \to \IET$ is conjugated to a standard torus action in one parameter.

\begin{defi} \label{standard torus action}
Let $\lambda = (\lambda_{1},\lambda_{2},...,\lambda_{n}) \in (0,1)^{n}$ with $\sum_{i} \lambda_{i} = 1$. Let $\beta_{i}$ the sum $\beta_{i} = \sum_{j \leq i} \lambda_{j}$, in particular \(\beta_{0} = 0\) and \(\beta_{n} = 1\). The standard torus action associated to $\lambda$ is the function $\tau_{\lambda}: \T^{n} \to \IET$ define by 

$$ [\tau_{\lambda}(\alpha)](x) = \left \{ \begin{matrix} 
x + \lambda_{j} \alpha_{j} & \mbox{if } x \in [\beta_{j-1}, \beta_{j} - \lambda_{j}\alpha_{j})\\
 x + \lambda_{j}\alpha_{j} - \beta_{j} & \mbox{if } x \in [\beta_{j} - \lambda_{j}\alpha_{j}, \beta_{j}) \end{matrix}\right.$$

An one parameter standard torus action is a map $\theta:\R \to \IET$ such that $\theta(t) = \tau_{\lambda}(t \alpha)$ for some $\alpha, \lambda$ in the previous conditions.
\end{defi}

In other words the standard torus action is a rotation in each interval $[\beta_{j},\beta_{j+1})$ of angle $\alpha_{j}$ (see figure 1).

\begin{figure}[ht]\centering   \label{example st torus action}
\includegraphics[scale=0.1]{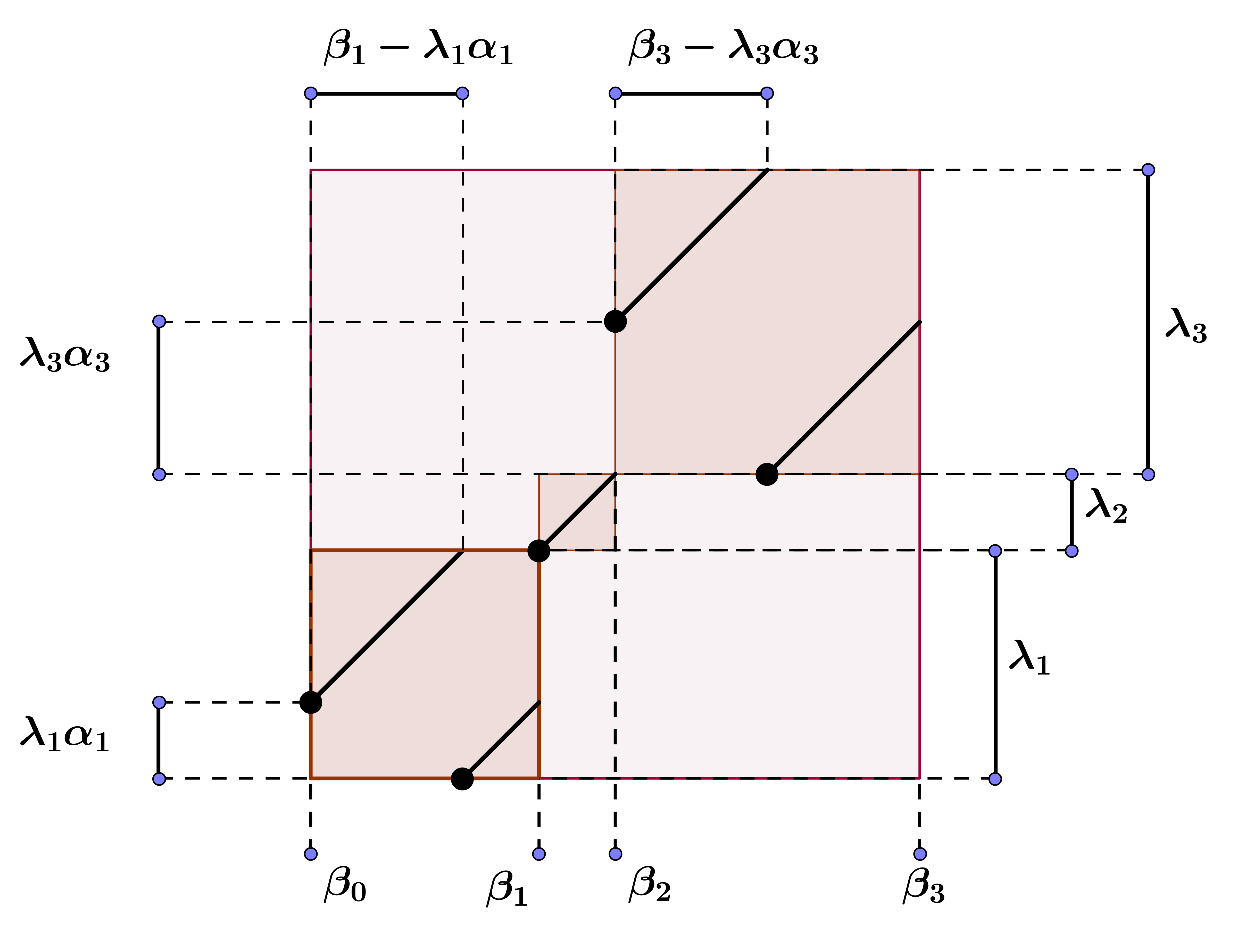}
\caption{Standard torus action (in this case $\alpha_{1} = \frac{1}{3}$, $\alpha_{2} = 0$ and $\alpha_{3} = \frac{1}{2}$) }
\end{figure}

To study continuity in our context we will provide a distance in \(\ACI\).

\begin{defi} \label{distance in ACI}
We define $\tilde{d}_{1}, d :\ACI \times \ACI \to \R^{+} \cup \{0\}$ by $$\tilde{d}_{1}(f,g) = \int_{\S^{1}} \d(f(x),g(x)) d\mu , \hspace{0.5cm}  d(f,g)= \tilde{d}_{1}(f,g) +  \tilde{d}_{1}(f^{-1},g^{-1})$$
where \(\d\) is the distance in \(\S^{1}\), and \(\mu\) the measure generated by \(\d\) (Notation \ref{distance S1}).
\end{defi}

We study the distance $d$ in the section 3. In particular we will prove the following result.

\begin{prop} \label{ACI is a topological group}
The group $(\ACI, d)$ is a topological group.
\end{prop}

Note that the subgroup \(\AIET(\S^{1}) = \langle \IET(\S^{1}), \PL_{+}(\S^{1}) \rangle < \ACI\) verifies that \((\AIET(\S^{1}),d)\) is also a topological group.

The topology induced in \(\IET\) by \(d\), and \(\tilde{d}_{1}\) verifies the hypothesis of Novak's theorem.

The main result of \cite{Novak1} is not valid in \((\ACI,d)\), for example via differential flow we can construct \(\rho:\R \to \ACI\) with a finite numbers of fix points  (Example \ref{ejemplo con dos discontinuidades en un intervalo}). 

Also this result is not valid in \((\AIET(\S^{1}),d)\). More precisely the existence of exotic circles in \(\PL_{+}(S^{1})\) (\cite{Minakawa}) implies the existence of continuous homomorphisms \(\rho: \R \to (\AIET(\S^{1}),d)\) such that are not conjugated to a standard torus action. 

But we can adapt this result to our context. This theorem can be interpreted in terms of domains. We will use the notation of \cite{Dahmani1} to work in domains other than [0,1).

\begin{defi} \label{inclusion in domain}
A domain \(\mathcal{D}\) is a disjoint finite union of circles \(\R /l\Z\) and semi-open bounded intervals \([\alpha,\beta)\).
Each component of \(\mathcal{D}\) has a natural metric and orientation. 

We consider any domain \(\mathcal{D}\) with an order in his components and denote by \(\mathcal{I}_{j}\) the canonical inclusion of \(j\)-th component in \(\mathcal{D}\).

An interval exchange transformation \(h: \mathcal{D} \to \mathcal{D}^{\prime}\) between two domains \(\mathcal{D},\, \mathcal{D}^{\prime}\) is a bijection function which is piecewise isometric, orientation-preserving, right continuous and with only finitely many discontinuity points.
The set of all interval exchange transformation between  \(\mathcal{D},\, \mathcal{D}^{\prime}\) is denote by \(\IET(\mathcal{D},\mathcal{D}^{\prime})\). Also we denote \(\IET(\mathcal{D}) = \IET(\mathcal{D},\mathcal{D})\).
\end{defi}

Note that by the previous definition \(\IET([0,1)) = \IET\).

In this article we work with domains whose components are circles.

The result of \cite{Novak1}  can be viewed as: for any continuous homomorphism $\tau: \R \to \IET$ there exist a domain \(\mathcal{D}\) and \(f \in \IET(\S^{1},\mathcal{D})\) such that \(f \circ \tau(t) \circ f^{-1} \in \IET(\mathcal{D})\) is continuous for any \(t \in \R\).

We adapt this result to the group \((\ACI,d)\). In particular this show that this property does not need the piecewise isometric condition.

\begin{teo} \label{main theorem}
Let $\rho:\R \to \ACI$ a continuous homomorphism. Then there exist a domain $\mathcal{D}$ and an interval exchange transformation $f \in \IET(\S^{1},\mathcal{D})$ such that the map $\hat{\rho}:\R \to \mathcal{P}\AC_{+}(\mathcal{D})$ defined by $\hat{\rho}(t) = f \circ \rho(t) \circ f^{-1}$ is a continuous homomorphism and verifies \(\hat{\rho}(t)\) is continuous for any \(t\in \R\).
\end{teo}

We prove this theorem in section 5,

In the section 6 we exhibit several examples to analyze the hypotheses of the main theorem.

\section{Preliminaries and Notations}

\subsection{Notations for circles}

\begin{nota} \label{distance S1}
We denote $(\S^{1})_{l} = \R/l\Z$, and $\pi_{l}:\R\to(\S^{1})_{l}$ the canonical projection.
Let $(\d)_{l}$ the distance induced by $(\pi)_{l}$ in \((\S^{1})_{l}\) and $(\mu)_{l}$ the measure generated in $(\S^{1})_{l}$ by $(\d)_{l}$.
In the case of unit circle we omit the sub index \(l=1\).
\end{nota}

We define the intervals in \((\S^{1})_{l}\) with the counterclockwise orientation. 

\begin{nota} \label{notation intervals}
Let $a,b \in (\S^{1})_{l}$ with $a \neq b$, we denote by $[a,b] = \pi_{l}([\hat{a},\hat{b}])$, where $\pi_{l}(\hat{a}) = a,\ \pi_{l}(\hat{b}) = b$ and $\pi_{l}([\hat{a},\hat{b}]) \subsetneq (\S^{1})_{l}$.
We will use the notation $(a,b)$, $[a,b)$ and $(a,b]$ in the same way.

For any interval $I = [a,b) \subset \R$, with $b - a = l$ we will denote $\iota_{I}:I \to (\S^{1})_{l}$ the canonical bijection (i.e. $\iota_{I}(x) = \pi_{l}(x - a)$). 

Note that \(\iota_{I} \in \IET(I,(\S^{1})_{l})\) and determines, a canonical bijection \(\tilde{\iota}_{I}: \IET(I) \to \IET((\S^{1})_{l})\) define by \(\tilde{\iota}(g)(x) = (\iota_{I} \circ g \circ (\iota_{I})^{-1})(x)\).
\end{nota}

The measure of intervals and the distance of their endpoints are related in the obvious way.

\begin{obs} \label{interval measure-distance endpoints}
Let $a,b \in \S^{1}$. Then $diam((a,b)) = \mu((a,b))$ in particular we have that $\d(a,b) = \mu([a,b])$ or $\d(a,b) = \mu([b,a])$.
Also, if  $\mu([a,b]) \leq 1/2$ then $\d(a,b) = \mu([a,b])$.
\end{obs}

\begin{defi} \label{half-interval}
We define the right half-interval of an interval \([a,b] \subsetneq \S^1$ as $[a,b]_{+} = \pi\left(\left[\frac{\hat{a}+ \hat{b}}{2},\hat{b}\right]\right)$, where $\pi(\hat{a}) =a, \pi(\hat{b}) = b$ and $\pi([\hat{a},\hat{b}]) \neq \S^{1}$.
Also, we define the left half-interval of $[a,b]$ as \\ $[a,b]_{-} = \pi\left(\left[\hat{a},\frac{\hat{a}+ \hat{b}}{2}\right]\right)$.

We will use $(a,b)_{+}$, $(a,b)_{-}$, $[a,b)_{+}$, $[a,b)_{-}$, $(a,b]_{-}$ and $(a,b]_{+}$ analogously. 
\end{defi}

\begin{figure}\centering
\includegraphics[scale = 0.6]{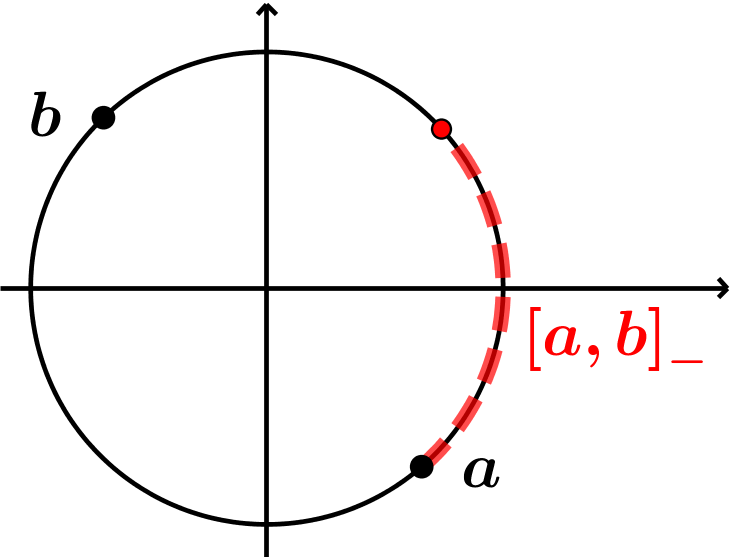}
\caption{Interval representation}
\end{figure}

Now we can define right continuous and absolutely continuous in the circles \((S^{1})_{l}\).

\begin{defi}\label{def rlim rcont abs cont}
Let \(f:(\S^{1})_{l} \to (\S^{1})_{l}\), \([a,b) \subsetneq (\S ^{1})_{l}\) and \(\hat{a},\, \hat{b} \in \R\) such that \(\pi_{l}(\hat{a}) = a,\, \pi_{l}(\hat{b}) = b\).
We say that:
\begin{enumerate}
\item The function \(f\) has right (left) limit in \(a\) if there exist \(c \in (\S^{1})_{l}\), \(\hat{c} \in \R\),  and \(\delta > 0\) such that \(\pi_{l}(\hat{c}) = c\), and \(((\iota_{l})\vert_{[\hat{c},\hat{c}+l)})^{-1} \circ f \circ (\iota_{l})\vert_{[\hat{a}-\delta,\hat{a} + \delta)}\) has right (left) limit in \(\hat{a}\). Also we say \(f\) is right (left) continuous in a if \(((\iota_{l})\vert_{[\hat{c},\hat{c}+l)})^{-1} \circ f \circ (\iota_{l})\vert_{[\hat{a}-\delta,\hat{a} + \delta)}\) is right (left) continuous in \(\hat{a}\).
\item The function \(f\) is absolutely continuous in \([a,b)\) if there exists \(\hat{c} \) such that \\ \(((\iota_{l})\vert_{[\hat{c},\hat{c}+l)})^{-1} \circ f \circ (\iota_{l})\vert_{[\hat{a},\hat{a} + l)}\) is absolutely continuous in \([\hat{a},\hat{b})\).
\end{enumerate}
\end{defi}

Since \(f \in \ACI\) preserves orientation in the intervals of continuity, it is right continuous, there is a finite  discontinuity points and is a bijection, then the definition of right and left limits can be expressed in terms of balls.

\begin{nota}
Let $p \in \S^{1}$ and $\delta \in (0,\frac{1}{2})$ we denote the ball of center $p$ and radius $\delta$ by $B(p,\delta) = \{x \in \S^{1} : \d(x,p) < \delta\}$. 
Since $B(p,\delta) = (a,b)$ for some $a,b \in \S^{1}$ we define $B(p,\delta)_{+} = [a,b)_{+}$, $B^{*}(p,\delta)_{+} = (a,b)_{+}$, $B(p,\delta)_{-} = (a,b]_{-}$ and $B^{*}(p,\delta)_{-} = (a,b)_{-}$.
\end{nota}

\begin{obs}
Let \(f \in \ACI\) and \(a \in (\S^{1})_{l}\).

The right limit verifies \(\displaystyle \lim_{x \to a^{+}} f(x) = b\) if only if for any \(\epsilon > 0\) exists \(\delta > 0\) such that \(f(B(a,\delta)_{+}) \subset B(b,\epsilon)_{+}\).

The left limit verifies \(\displaystyle \lim_{x \to a^{-}} f(x) = b\), if only if for any \(\epsilon > 0\) exists \(\delta > 0\) such that \(f(B^{*}(a,\delta)_{-}) \subset B^{*}(b,\epsilon)_{-}\).
\end{obs}

We define a distance in domains in the same way as in circles.

\begin{defi}
For a domain \(D = \uplus_{i \leq n} (\S^{1})_{\lambda_{i}}\) we define the distance \(\d_{D}\) by 

\(\d_{D}(x,y)= \left\{ \begin{matrix}
\d_{\lambda_{j}}(x,y) & \text{ if there exists } j \text{ such that } x,y \in (\S^{1})_{\lambda_{j}}, \\
\sum_{i \leq n} \lambda_{i} & \text{ in other case. }
\end{matrix} \right.\)
\end{defi}

\subsection{Discontinuity points of the maps of $\ACI$}

Recall that in this work any domain is an union of circles.

From now on, for the sake of simplicity we denote the points \(x\) of \(\S^{1}\) with his canonical identification in \([0,1)\) via \(\iota^{-1}\).

The points of discontinuity are crucial for proving several results in this work.
In this subsection we will study some properties and notation associated to them.

Since any map $f \in \ACI$ preserves orientation in the intervals of continuity (item 4 of Definition \ref{interval def of ACI}) the order of points in each one is preserved.

\begin{defi}
Let $k \in \N, k \geq 3$, a $k$-tuple in $\S^{1}$, $v = (x_{1},...,x_{k})$ is a ordered k-tuple if and only if for all 3-tuple $(i_{1},i_{2},i_{3})\in \N^{3}$ with $1 \leq i_{1} < i_{2} < i_{3} \leq k$ we have $x_{i_{2}} \in (x_{i_{1}},x_{i_{3}})$.
\end{defi}

\begin{obs} \label{continuity in ordered tuples}
Let $f \in \ACI$.
\begin{enumerate}
\item If $(x_{1},x_{2},x_{3})$ is an ordered $3$-tuple and $f$ is continuous on $[x_{1},x_{3}]$ then $(f(x_{1}),f(x_{2}),f(x_{3}))$ is an ordered 3-tuple.
\item Applying the previous item to $f^{-1}$, if $(f(x_{1}),f(x_{2}),f(x_{3}))$ is an ordered 3-tuple and $(x_{1},x_{2},x_{3})$ is not, then $f^{-1}$ is not continuous in $[f(x_{1}),f(x_{3})]$ so $f$ is not continuous in $[x_{1},x_{3}]$.
\item In summary $f$ is continuous in $[a,b)$ if and only if $f^{-1}$ is continuous in $[f(a),f(b))$.
\end{enumerate}
\end{obs}

\begin{defi} \label{Notation BP_0 and delta}
Let \(\mathcal{D}, \mathcal{D}^{\prime}\) be two domains and \(f:\mathcal{D} \to \mathcal{D}^{\prime}\) a function  with finite number of discontinuity points. We define $BP_{0}(f) = \{ x \in \mathcal{D} :  f \mbox{ is not continuous at } x \}$, \(\sharp f = \# BP_{0}(f)\) and $Cont(f) = \mathcal{D} \setminus BP_{0}(f)$. 

Also, for $f \in \ACI$, we define:
\begin{itemize}
\item  $\delta_{f}(p) = \min\{1,\d(BP_{0}(f) \setminus \{ p \},p)\}$ and $\delta_{f} = \min (\{1,\delta_{f}(p) : p \in BP_{0}(f)  \} )$.
\item In the case of \(\mathcal{D} = \S^{1}\) and $BP_{0}(f) \neq \emptyset$, we use the following notation  $BP_{0}(f) = \{p_{1},...,p_{k}\}$ where $(p_{1},...,p_{k})$ is an ordered k-tuple and $0 \in (p_{k},p_{1}]$. Let $I_{f}(i) = [p_{i},p_{i+1})$ for $i < k$ and $I_{f}(k) = [p_{k},p_{1})$ (see figure 2).
The intervals $I_{f}(j)$ are maximal intervals of continuity of $f$. In the case that $f$ is a continuous function we will denote $I_{f}(0)= \S^{1}$.
\end{itemize}

\begin{figure}\centering
\includegraphics{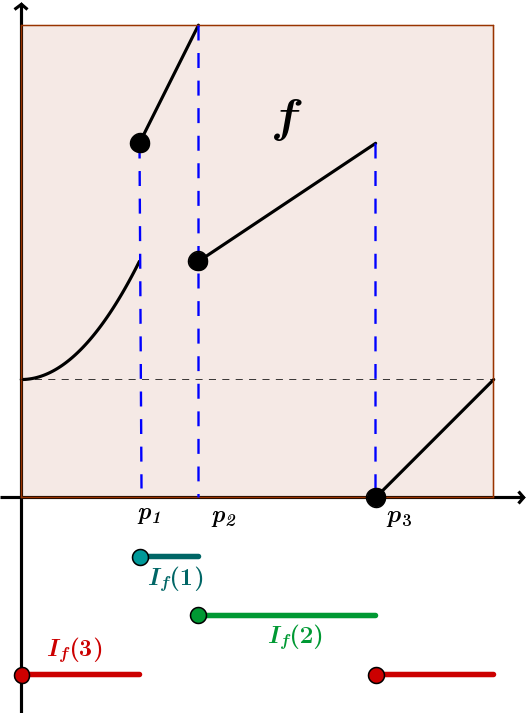}
\caption{Example of a map $f \in \ACI$}
\end{figure}

\end{defi}

\begin{obs}
Let $f \in \mathcal{PAC}_{+}((\S^{1})_{l})$ a discontinuous function. Then $3 \leq \sharp f$ and $0 < \delta_{f} < l$.
\end{obs}

The composition of continuous functions is continuous, so we can relate $BP_{0}(f)$ and $BP_{0}(g)$ with $BP_{0}(f \circ g)$.

\begin{obs} \label{Discontinuidad de la inversa}
Let \(\mathcal{D}, \mathcal{D}^{\prime}, \mathcal{D}^{\prime \prime} \) be domains and $g:\mathcal{D} \to \mathcal{D}^{\prime}, f:\mathcal{D}^{\prime} \to \mathcal{D}^{\prime \prime}$ bijections with finite discontinuity points.
Then:
\begin{enumerate}
\item $BP_{0}(f \circ g) \subset BP_{0}(g) \cup g^{-1}(BP_{0}(f))$.
\item If $x \in BP_{0}(g)$ and $g(x) \in \mbox{Cont}(f)$, then $x \in BP_{0}(f \circ g)$. Also if $x \in \mbox{Cont}(g)$ and $g(x) \in BP_{0}(f)$, then  $x \in BP_{0}(f \circ g)$.

In summary, $$\left[BP_{0}(g) \cap g^{-1}(Cont(f)) \right] \cup \left[Cont(g) \cap g^{-1}(BP_{0}(f)) \right]  \subset BP_{0}(f \circ g).$$

Then, the cardinal of $BP_{0}(f \circ g)$ verifies the follow inequality 
$$ \vert\#  BP_{0}(f) - \# BP_{0}(g) \vert \leq \#BP_{0}(f \circ  g) \leq \# BP_{0}(f) + \# BP_{0}(g).$$

\item In the case $g = f^{-1}$ it holds $\# BP_{0}(f)  =  \# BP_{0}(f^{-1})$. More precisely $BP_{0}(f^{-1}) = f(BP_{0}(f))$.

\end{enumerate} 
\end{obs}

The last claim of the remark 2.10 can be adapted to $I_{f}(i)$.

\begin{obs}\label{left limit in BP part 1}
Since $I_{f}(i)$ is a maximal interval of continuity of $f$, then $f(I_{f}(i))= I_{f^{-1}}(j)$ for some j. Moreover $\lim_{x \to b^{-}} f(x) \in BP_{0}(f^{-1})$ for any $b \in BP_{0}(f)$.
\end{obs}
 
\begin{obs} \label{acotacion de conjuntos por h}
Let $f \in \ACI$ and $\delta > 0$. Then there exists $\delta_{0} \in (0,1/2)$ such that for any countable union of intervals $A \subset \S^{1}$  with $\mu(A) \leq \delta_{0}$ then $\mu(f(A)) \leq \delta$.
\end{obs}

Finally, given $f \in \ACI$ we consider a neighbourhood of $BP_0(f)$.

\begin{defi} \label{Vn}
Let $f \in \ACI$, we define $V_{n}(f) = \cup_{p \in BP_{0}(f)} B(p,\delta)$, where $\delta = \frac{1}{2n(\sharp f + 1)}$.
\end{defi}

\begin{obs}
For any $f \in \ACI$ and $n \in \N \setminus \{0\}$ it holds that $\mu(V_{n}(f)) \leq 1/n$, $\S^{1} \setminus V_{n}(f)$ is a closed set, and $f\vert_{\S^{1} \setminus V_{n}(f)}$ is absolutely continuous and, in particular, uniformly continuous.
\end{obs}

%\begin{prop} 
%Let $h \in \ACI$ and $\delta > 0$. There exists $\delta_{0} \in (0,1/2)$ such that $\forall A \subset \S^{1}$, countable union of intervals with $\mu(A) \leq \delta_{0}$ then $\mu(h(A)) \leq \delta$.
%\end{prop}

%\begin{proof}
%In the case of $h \in C^{0}(\S^{1})$, the result is inmediate because $\S^{1}$ is compact. let see the case $BP_{0}(h) \neq \emptyset$.

%Because $h(V_{n}(h))$ is a decrease sequence and $h$ is a bijection then $\cap_{n} h(V_{n}(h)) = h(BP_{0}(h))$. Apply continuity of measure, exists $n_{0}$ such that $\mu(h(V_{n}(h))) \leq \delta/2$.

%For any $I$, conected component of $\S^{1} \setminus V_{n}(h)$, we have $h$ is absolutely continuous.

%Applying \ref{interval measure-distance endpoints} for any $B = [b_{1},b_{2}] \subset I$ we have $\mu(B) = \d(b_{1},b_{2})$ and $\mu(h(B)) = \d(h(b_{1}),h(b_{2}))$ .Then exists $\delta_{0} > 0$ such for any $B \subset I$ countable union of intervals with $\mu(B) \leq \frac{ \delta_{0}}{\#BP_{0}(h)}$, we have $\mu(h(B)) < \frac{\delta}{2\#BP_{0}(h)}$.

%Finally, separete $A$ in this set we have
%$$\mu(h(A)) \leq \mu(h(A \cap V_{n}(h))) + \mu(h(A \setminus V_{n}(h))) %\leq  \mu(h(V_{n}(h))) + \mu(h(A \setminus V_{n}(h))) \leq \delta.$$ 
%\end{proof}

\section{$\ACI$ is a topological group}

In this section we prove that $(\ACI,d)$ is a topological group (Proposition \ref{ACI is a topological group}), and we study basic topological properties of it.

In \cite{Novak1}, Novak considered all topological group structures that satisfy some properties.
These topologies verify the following naive statement. For $f, g \in \IET$, if $\sharp f = \sharp g$, $BP_{0}(f)$ is close to $BP_{0}(g)$ and \(\sigma_{f} = \sigma_{g}\) (i.e. has the same rearrangement permutation) then $f$ is close to $g$. Since the elements of \(\IET\) are piecewise isometric, this property is closed by inverse. But this is not true in \(\ACI\).

The distance $\tilde{d}_{1}$ (Definition \ref{distance in ACI}) restricted to \(\IET\) (distance $\underline{d}$ defined in the Proposition 2.2 of \cite{Novak1}), verifies this property. However  $\tilde{d}_{1}$ is not a good distance in \(\ACI\).

The following example shows that $(\ACI, \tilde{d}_{1})$ is not a topological group. Also the restriction to \(\AIET(\S^{1})\) is not a topological group.

\begin{ej}
For $n \in \N$ let denote $a_{n}(k) = \frac{1}{4} + \frac{k}{4n}$ and $b_{n}(k) = 1 - \frac{1}{4n} + \frac{k}{4n^{2}}$. 
Note that $a_{n}(k+1) - a_{n}(k) = \frac{1}{4n}$, $\cup_{k < n}[a_{n}(k),a_{n}(k+1)) = [\frac{1}{4},\frac{1}{2})$, $b_{n}(k+ 1) - b_{n}(k)  = \frac{1}{4n^{2}}$ and $\cup_{k < n}[b_{n}(k),b_{n}(k+1)) = [1 - \frac{1}{4n},1)$

Let $f_{n} \in \AIET(\S^{1})$ a sequence defined by
$$f_{n}(x) = \left \{ \begin{matrix} 
x & \mbox{ if } x \in \left[0,\frac{1}{4}\right) 
\\ \frac{1}{2}(x - a_{n}(k)) +a_{n}(k) & \mbox{ if } x \in [a_{n}(k),  a_{n}(k+1)) \mbox{ with } k < n
\\ \frac{2n}{2n-1}(x - \frac{1}{2}) + \frac{1}{2} & \mbox{ if } x \in \left[\frac{1}{2},1-\frac{1}{4n}\right)
\\ \frac{n}{2}(x - b_{n}(k)) + a_{n}(k) + \frac{1}{8n}& \mbox{ if } x \in [b_{n}(k),b_{n}(k+1)) \mbox{ with } k < n
 \end{matrix}\right.$$

\begin{figure}[h]\centering
\includegraphics[scale = 0.7]{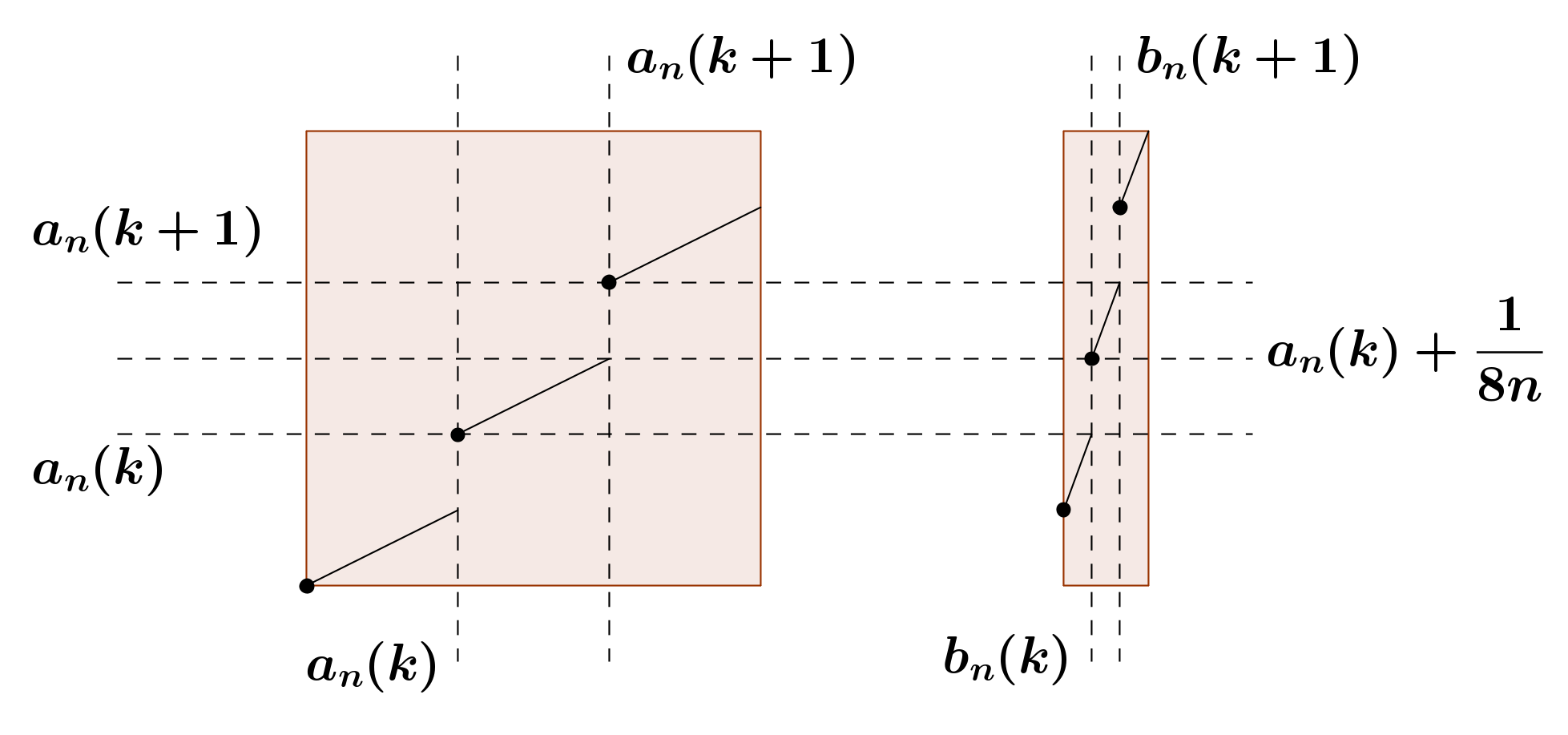}
\caption{Representation of maps \(f_{n}\) of Example 3.1.}
\end{figure}

It is easy to check that $\d(f_{n}(x),x) \leq \frac{1}{4n},\ \forall x < 1 - \frac{1}{4n}$, then $\tilde{d}_{1}(f_{n},id) \leq \frac{1}{2n}$, in particular $\tilde{d}_{1}(f_{n},id) \to 0$. 

On the other hand, $\d(f_{n}^{-1}(x),x) \geq \frac{1}{4}$ for all $x \in f_{n}([1-\frac{1}{4n},1))$, since $\mu(f_{n}([1-\frac{1}{4n},1)) = \frac{1}{8}$ so $\tilde{d}_{1}(f_{n}^{-1},id) \geq \frac{1}{24}$.
\end{ej}

This example shows that \((\AIET(\S^{1}),\tilde{d}_{1})\) is not topological group.
The reason is that small intervals can be transformed in uniformly big ones.
This fact does not happen in \(\IET\) because its elements are piecewise isometric. 
%Let us now show that $d$ be have better than $\tilde{d}_{1}$ in relation to the change of measure of the intervals for maps close to identity.

\begin{prop} \label{function close to identity chenge little the measure}
Let $\delta > 0$. Then there exists $\epsilon > 0$ such that $\vert \mu(I) - \mu(g(I)) \vert \leq \delta$ for any interval $I \subset \S^{1}$ and $g \in B(id, \epsilon)$.
\end{prop}

\begin{proof}
Let $I$ be an interval, and $g \in \ACI$. We can assume that $\mu(I) \leq \frac{1}{2}$.

Let $\delta_{0} = \mu(I),$ $\delta_{1} = \mu(g(I))$ and $\tilde{\delta} = \vert \delta_{1} - \delta_{0} \vert$. We study two cases: $\delta_{1} > \delta_{0}$ and $\delta_{1} < \delta_{0}$.

Case 1: $\delta_{1} > \delta_{0}$.

Let $I_{1} = g(I) \setminus B(I,\frac{\tilde{\delta}}{4})$. Note that $I_{1}$ is a finite union of intervals and $\mu(I_{1}) \geq \frac{\tilde{\delta}}{2}$.

Also $\d(g^{-1}(x),x) \geq \frac{\tilde{\delta}}{4}$ for all $x \in I_{1}$. Then $d(g,id) \geq \int_{I_{1}} \d(g^{-1}(x),x) \geq \frac{\tilde{\delta}^{2}}{8}$.

Case 2: $\delta_{1} < \delta_{0}$.

Let $a \in \S^{1}$ be such that $I \subset \overline{B(a,\frac{\delta_{0}}{2})} )$ (i.e. $a$ is the center of $I$).

Let $I_{1} = B(a,\frac{\delta_{0}}{2} - \frac{\tilde{\delta}}{4}) \setminus g(I)$. Note that $I_{1}$ is a finite union of intervals and $\mu(I_{1}) \geq \frac{\tilde{\delta}}{2}$, and $\d(g^{-1}(x),x) \geq \frac{\tilde{\delta}}{4}$ for all $x \in I_{1}$. Then we have that $d(g,id) \geq \int_{I_{1}} \d(g^{-1}(x),x) \geq \frac{\tilde{\delta}^{2}}{8}$.

In conclusion $d(g,id) > \frac{\vert \mu(I) - \mu(g(I)) \vert^{2}}{8}$, then $\epsilon = \frac{\delta^{2}}{16}$ verifies the thesis.
\end{proof}

%Some properties can be view in an intuitive way with $\tilde{d}_{1}$. 
%
%\begin{prop} \label{intervalos se quedan cerca (d0)}
%Let \(\delta > 0\). There exists \(\epsilon > 0\) such that for any \(I\) interval of \(\S^{1}\) with \(\mu(I) = L \geq \delta\), if \(\tilde{d}_{1}(f,id) \leq \epsilon\), then \(\mu(f^{-1}(I) \cap I) \geq \frac{L}{8}\).
%\end{prop}
%
%\begin{proof}
%
%Let \(\hat{I} = \{x \in I : B(x,\delta/4) \subset I \}\).
%
%Let \(f \in \ACI\) and \(\hat{J} = \{x \in \hat{I} \text{ such that } f(x) \notin I \}\). 
%Notice that \(d_{0}(f,id) = \int_{\S^{1}} \d(f(x),x) \geq \int_{\hat{J}} \d(f(x),x) \geq \mu(\hat{J})\frac{\delta}{4}\).
%
%Let \(\epsilon = \frac{\delta^{2}}{16}\), if \(d_{0}(f,id) \leq \epsilon\) then \(\mu(\hat{I} \setminus \hat{J}) \geq \frac{L}{2} - \frac{\delta}{4} \geq \frac{L}{4}\).
%Since \(f(\hat{I} \setminus \hat{J}) \subset I\) we can conclude the thesis.
%\end{proof}
%
%The same result is valid in the distance \(d\).

To prove the Proposition \ref{ACI is a topological group} we have to show that the right action, the left action and the inverse are continuous.

By symmetry of the definition of $d$, the continuity of the inverse map is trivial.

\begin{obs} \label{inverse in ACI is continuous}
The inverse map $f \mapsto f^{-1}$ is an isometry in $(\ACI,d)$. 
\end{obs}

For $h \in \ACI$, we define the functions $L_{h}, R_{h}: \ACI \to \ACI$ by $L_{h}(f) = h \circ f$, $R_{h}(f) = f \circ h$. 

The function $L_{h}$ is continuous if and only if for any sequence $\{f_{n}\}$ such that $d(f_{n},f) \to 0$ verifies $d(L_{h}(f_{n}),L_{h}(f)) \to 0$. By definition  $d(h\circ f_{n},h \circ f) = \tilde{d}_{1}(h \circ f_{n},h \circ f) + \tilde{d}_{1}(f_ {n}^{-1} \circ h^{-1}, f^{-1} \circ h^{-1})$, then we can study the continuity of $L_{h}$ and $R_{h}$ via $\tilde{d}_{1}$.

\begin{defi} Let $h \in \ACI$ we define $\tilde{L}_{h}, \tilde{R}_{h}:(\ACI,d) \to (\ACI,\tilde{d}_{1})$ by $\tilde{L}_{h}(f) = h \circ f$ and $\tilde{R}_{h}(f) = f \circ h$.
\end{defi}

\begin{obs} \label{continuity by d_0}
The function $L_{h}$ is continuous if and only if $\tilde{L}_{h}$ and $\tilde{R}_{h^{-1}}$ are continuous. The function $R_{h}$ is continuous if and only if $\tilde{R}_{h}$ and $\tilde{L}_{h^{-1}}$ are continuous.
\end{obs}

For $f,g \in \ACI$ with $f$ close to $g$ in $\tilde{d}_{1}$ 
we have that for most points \(x\), the points $f(x)$ is close to $g(x)$. So let us define the set where this fact does not happen.

\begin{defi}
Let $f,g \in \ACI$, we define the sets
\begin{itemize}
\item $U_{n}(f,g) = \{x \in \S^{1} : \d(f(x),g(x)) > n\tilde{d}_{1}(f,g)\}$, and
\item $U(f,g)[\delta] =\{ x \in \S^{1} : \d(f(x),g(x)) > \delta \}$.
\end{itemize}
Note that $U_{n}(f,g)$ and $U(f,g)[\delta]$ are countable union of intervals.
\end{defi}

\begin{obs} \label{desigualdad de U(f,g)}
Let $f,g \in \ACI$ then the following inequalities are verified
\begin{itemize}
\item $\mu(U_{n}(f,g)) \leq 1/n$
\item if $\mu(U(f,g)[\delta_{0}]) = \delta$ then $\tilde{d}_{1}(f,g) \leq \frac{1}{2} \delta  + \delta_{0}(1- \delta) \leq \frac{\delta}{2} + \delta_{0}$ 
\end{itemize}

\end{obs}

These sets allow us describe the convergence in $\tilde{d}_{1}$.

\begin{prop} \label{convergence by sets}
Let $f, \{f_{n}\} \in \ACI$. The following statements are equivalent.
\begin{enumerate}
\item $\tilde{d}_{1}(f_{n},f) \to 0$.
\item For any $\delta > 0,\ \mu(U(f_{n},f)[\delta]) \to 0$.
\item For any  $\delta > 0,\ \exists N \in\N$ such that $\mu(U(f_{n},f)[\delta]) \leq \delta,\,  \forall n > N$.
\end{enumerate}
\end{prop}

\begin{proof} \ 

(1 \(\Rightarrow\) 2) Let $\delta > 0$, there exists a sequence $m_{n} \in \N$ such that $m_{n} \to \infty$ and $m_{n}\tilde{d}_{1}(f_{n},f)   < \delta$. Since $U(f,f_{n})[\delta] \subset U_{m_{n}}(f,f_{n})$ and $\mu(U_{m_{n}}(f,f_{n})) \leq \frac{1}{m_{n}}$, then $\mu(U(f_{n},f)[\delta]) \to 0$.

(2 \(\Rightarrow\) 3) Immediate.

(3 \(\Rightarrow\) 1) Is equivalent to prove that fixed $\epsilon > 0$ there exists $N_{0}$ such that $\tilde{d}_{1}(f_{n},f) < \epsilon$ for all $n > N_{0}$. 
By hypothesis there exists $N_{0}$ such that for all $n > N_{0}$ we have that $U(f_{n},f)[\epsilon/2] \leq \epsilon/2$. Applying Remark \ref{desigualdad de U(f,g)} we conclude that $\tilde{d}_{1}(f_{n},f) \leq \frac{\delta}{2} + \delta < \epsilon$.
\end{proof}

Note that, in the case of $h \in C^{1}(\S^{1})$ applying substitution we have that $\tilde{R}_{h}$ is continuous. 

\begin{prop} \label{tilde omega is continuous}
Let $h \in \ACI$. The function $\tilde{R}_{h}$ is continuous.
\end{prop}

\begin{proof} Let $f,\{f_{n}\} \in \ACI$ such that $d(f_{n},f) \to 0$ let us prove that $\tilde{d}_{1}(f_{n} \circ h ,f \circ h) \to 0$. By Proposition \ref{convergence by sets}, it is enough to prove that for $\delta_{0} > 0$ there exists $N \in \N$ such that $\mu(U(f_{n} \circ h,f \circ h)[\delta_{0}]) \leq \delta_{0}$ for all $n > N$.

By Remark \ref{acotacion de conjuntos por h} there exists $\delta < \delta_{0}$ such that for any $A$ countable union of intervals with $\mu(A) < \delta $ we have that $\mu(h^{-1}(A)) \leq \delta_{0}$.

Let $N \in \N$ such that $\mu(U(f_{n},f)[\delta]) < \delta$, $\forall n > N$. Since $U(f_{n},f)[\delta]$ is a countable union of intervals and $h^{-1}(U(f_{n},f)[\delta]) = U(f_{n} \circ h , f \circ h)[\delta]$ then we conclude that
$$\mu(U(f_{n} \circ h, f \circ h)[\delta_{0}]) \leq  \mu(U(f_{n} \circ h, f \circ h)[\delta]) = \mu(h^{-1}(U(f_{n},f)[\delta])) \leq \delta_{0}$$
for any $n < N$.
\end{proof}

For the case of $\tilde{L}_{h}$ we begin with the $\mathcal{AC}_{+}(\S^{1})$ action, and we adapt this case to the general case.

\begin{lema} \label{continuity of varpi in C1}
Let $h \in \mathcal{AC}_{+}(\S^{1})$. The function $\tilde{L}_{h}$ is continuous.
\end{lema}

\begin{proof}

By Proposition \ref{convergence by sets} it is enough to prove that given a sequence $\{f_{n}\} \in \ACI$ and $\delta_{0} > 0$ such that $d(f_{n},f) \to 0$ there exists $N \in \N$ such that $\mu(U(h \circ f_{n},h \circ f)[\delta_{0}]) < \delta_{0}$, for all $n > N$.

The function $h$ is uniformly continuous then there exists $\delta_{1} \in (0,\delta_{0})$ such that $\d(h(x),h(y)) < \delta_{0}$ for all pair $x,y \in \S^{1}$ such that $\d(x,y) < \delta_{1}$, in particular $U(h \circ f_{n},h \circ f)[\delta_{0}] \subset U(f_{n},f)[\delta_{1}]$.

Let $N \in \N$ such that $\mu(U(f_{n},f)[\delta_{1}]) \leq \delta_{1} < \delta_{0}$ for all $n > N$. We conclude that $\mu(U(h \circ f_{n},h \circ f)[\delta_{0}]) < \delta_{0}$ for all $n > N$.
\end{proof}

To adapt this proof for the general case, we will use that $h\vert_{\S^{1} \setminus V_{n}(h)}$ is uniformly continuous, $\mu(f^{-1}( V_{n}(h)))$ and $\mu(f_{n}^{-1} (V_{n}(h)))$ are small.

\begin{prop} \label{continuidad de derecha composition}
Let $h \in\ACI$. The function $\tilde{L}_{h}$ is continuous.
\end{prop}

\begin{proof}
To verify the continuity of $\tilde{L}_{h}$, it is enough to prove that given a sequence $\{f_{n}\} \subset \ACI$ and $\delta_{0} > 0$ such that $d(f_{n},f) \to 0$ there exists $N \in \N$ such that $\mu(U(h \circ f_{n},h \circ f)[\delta_{0}]) < \delta_{0}$, for all $n > N$.

By the previous Lemma we can assume $BP_{0}(h) \neq \emptyset$.

Applying Remark \ref{acotacion de conjuntos por h} there exists $\delta_{1} > 0$ such that, for any countable union of intervals  $A$, with $\mu(A) < \delta_{1}$, it is verified that $\mu(f^{-1}(A)) < \frac{\delta_{0}}{8}$.

There exists $N_{0} \in \N$ such that $\frac{1}{N_{0}} < \min\{ \delta_{0},\delta_{1}\}$ and $V_{n}(h)$ is a union of $\sharp h$ disjoint intervals of measure $\frac{1}{n(\sharp h+1)}$ for all $n \geq N_{0}$. Note that $\mu(V_{n}(h)) < \frac{1}{n} \leq \frac{1}{N_{0}} <  \delta_{1}$ for any $n \geq N_{0}$, then $\mu(f^{-1}(V_{n}(h))) < \frac{\delta_{0}}{8}$. 

Since $\tilde{d}_{1}(f_{n},f) \to 0$ there exists $N_{1} > 2N_{0}$ such that $\mu\left(U(f_{n},f)\left[\frac{1}{4 N_{0}(\sharp h+1)}\right]\right) \leq \frac{1}{4 N_{0}(\sharp h+1)}$ for all $n \geq N_{1}$. 

Let $n \geq N_{1}$, if  $x \in f_{n}^{-1}(V_{n}(h)) \setminus f^{-1}(V_{N_{0}}(h))$ (equivalently \(\d(f_{n}(x),BP_{0}(h)) \leq \frac{1}{2n(\sharp h+1)}\) and \(\d(f(x),BP_{0}(h)) \geq \frac{1}{2N_{0}(\sharp h+1)}\)) we have $\d(f_{n}(x),f(x)) \geq \frac{1}{4 N_{0} (\sharp h+1)}$ so $x \in U(f_{n},f)\left[\frac{1}{4 N_{0}(\sharp h+1)}\right]$.
Then for all $n \geq N_{1}$ we have that $f_{n}^{-1}(V_{n}(h)) \subset f^{-1}(V_{N_{0}}(h)) \bigcup U(f_{n},f)\left[\frac{1}{4N_{0} (\sharp h+1)}\right]$ and $\mu(f_{n}^{-1}(V_{n}(h))) < \frac{\delta_{0}}{8} + \frac{1}{4N_{0} (\sharp h+1)} \leq \frac{\delta_{0}}{2}$.

Let $x, y \in \S^{1} \setminus  V_{N_{1}}(h)$, if $\d(x,y) \leq \frac{1}{2N_{1} (\sharp h+1)}$ then there exists $j$ such that $x, y \in I_{h}(j)$.

Since $h$ is uniformly continuous in $\S^{1} \setminus V_{N_{1}}(h)$ there exists $\delta_{2} \in \left(0,\frac{1}{4N_{1} (\sharp h+1)}\right)$  such that if $x,y \in \S^{1} \setminus V_{N_{1}}(h)$ and $\d(x,y) < \delta_{2}$ then $\d(h(x),h(y)) < \delta_{0}$.

In summary, if $x$ verify $\d(h \circ f(x) ,h \circ f_{n}(x)) > \delta_{0}$ then it satisfies one of the following properties
\begin{itemize}
\item $x \in f^{-1}(V_{N_{1}}(h))$
\item $x \in f^{-1}_{n}(V_{N_{1}}(h))$
\item $x \in U(f,f_{n})[\delta_{2}]$
\end{itemize}

Then $$U(h \circ f_{n},h \circ f)[\delta_{0}] \subset (U(f_{n},f)[\delta_{2}]) \cup f^{-1}(V_{N_{1}}(h)) \cup f_{n}^{-1}(V_{N_{1}}(h)) $$
so exists $N \in \N$ such that  $\mu(U(h \circ f_{n},h \circ f)[\delta_{0}]) \leq \frac{\delta_{0}}{8} + \frac{\delta_{0}}{8} + \frac{\delta_{0}}{2} < \delta_{0}$ for all $n > N$.
\end{proof}

Then, we have that $(\ACI,d)$ is a topological group. 

\begin{defi}
Let \(l \in \R^{+}\) and \(\mathcal{D}\) a domain. 
We define  \(\mathcal{PAC}_{+}((\S^{1})_{l})\) and \(\mathcal{PAC}_{+}(\mathcal{D})\) in the same way as \(\mathcal{PAC}_{+}(\S^{1})\).
Also we define the distance \(d_{l}\) in \(\mathcal{PAC}_{+}((\S^{1})_{l})\) and \(d_{D}\) in \(\mathcal{PAC}_{+}(\mathcal{D})\) in the same way as \(d\) in \(\mathcal{PAC}_{+}(\S^{1})\).
\end{defi}

Note that \((\mathcal{PAC}_{+}((\S^{1})_{l}),d_{l})\) and \((\mathcal{PAC}_{+}(\mathcal{D}),d_{D})\) are topological groups.

\begin{obs} \label{conjugar es continuo}
Let \(f \in IET(\mathcal{D},\mathcal{D}^{\prime})\).
The homomorphism \(\hat{f}:\mathcal{PAC}_{+}(\mathcal{D}) \to \mathcal{PAC}_{+}(\mathcal{D}^{\prime})\) defined by \(\hat{f}(g) = f \circ g \circ f^{-1}\) is a homeomorphism.
The inverse of \(\hat{f}\) is the map \(\widehat{f^{-1}}\).
\end{obs}

\section{Topological properties asociated to \(BP_{0}\) and \(\sharp\).}
Let us see now how the set $BP_{0}$ and the function $\sharp$ (Definition \ref{Notation BP_0 and delta}) are related with the distance \(d\).
For the sake of simplicity the results are proved in \(\S^{1}\) but are valid for any domains.

The function $\sharp$ is not locally bounded in any neighbourhood of $id$. 

\begin{ej}\label{sharp is not bounded close to id in IET}
Let $a_{n}:\{0,...,2^{n}\} \to [0,1]$,  $b_{n}:\{0,...,2^{n}-1\} \to [0,1)$ be sequences of functions defined by $a_{n}(k) = \frac{k}{2^{n}}$ and \(b_{n}(k) =\frac{a_{n}(k) + a_{n}(k+1)}{2}\) (i.e.: \(a_{n}\) is an equipartition of \([0,1)\) and \(b_{n}\) the middle point of the $k$-th interval).

Let $f_{n} \in \ACI$ be the function define by
 $$f_{n}(x) = \left \{ \begin{matrix} x + \frac{1}{2^{n+1}} & \mbox{ if } x \in [a_{n}(k),b_{n}(k))
\\ x - \frac{1}{2^{n+1}} & \mbox{ if } x \in [b_{n}(k),a_{n}(k+1))\end{matrix}\right.$$

Note that $d(f_{n},id) = 1/2^{n}$, so $f_{n} \to id$. Also for $n \geq 1$ we have $\sharp (f_{n}) = 2^{n+1}$, then $\sharp(f_{n}) \to \infty$.
\end{ej}

\begin{figure}[h]\centering
\includegraphics[scale = 0.7]{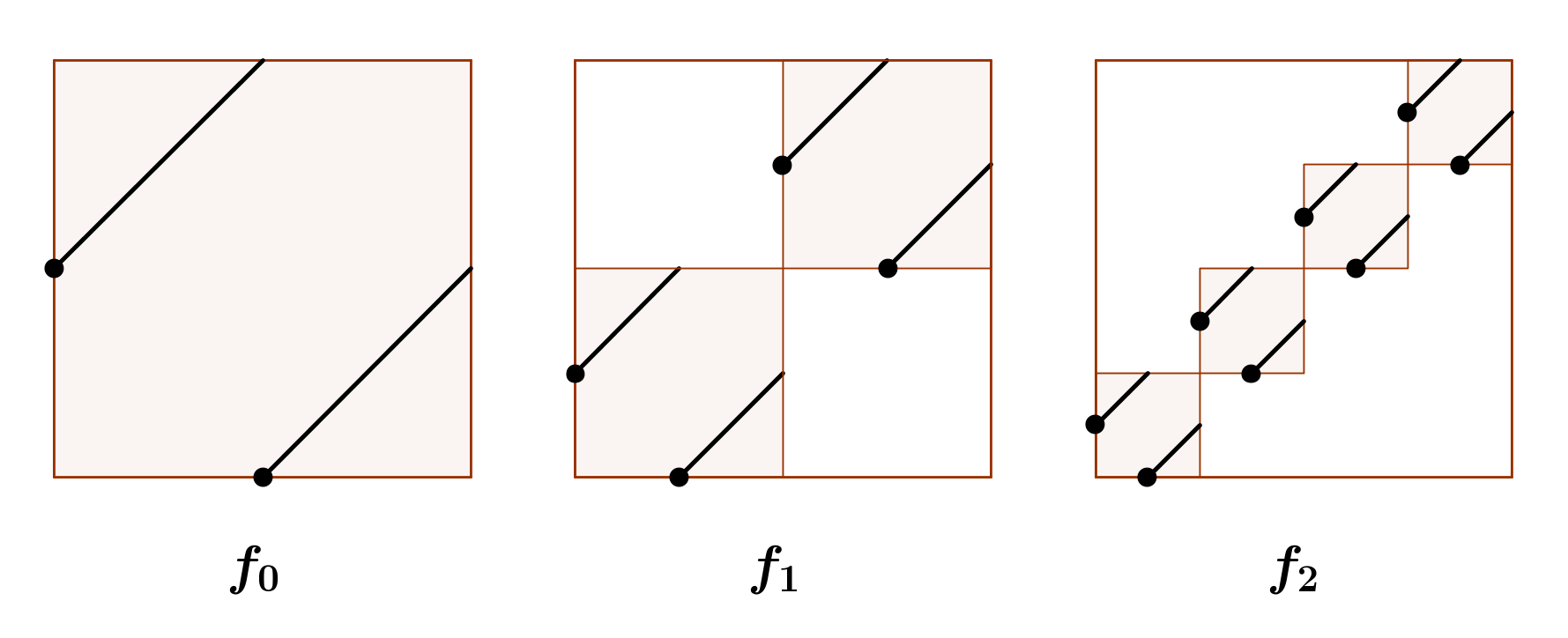}
\caption{First elements of $\{f_{n}\}$}
\end{figure}

By composition, this results is obtained in any function \(f \in \ACI\).

\begin{obs}
Let $f \in \ACI$ and $\epsilon > 0$, then $\sharp\vert_{B(f,\epsilon)}$ is not bounded.
\end{obs}

The previous result are valid in IET for any topology \(\tau\) in the hypotesis of the main theorem of \cite{Novak1} (the Example \ref{sharp is not bounded close to id in IET}, with small rotation in \([a_{n}(k),a_{n}(k+1))\) can be adapted to construct a converegence sequence to id in \(\tau\)).

Any discontinuity point of a map $f \in \ACI$ persist for a small perturbation. More precisely we have the following result. 

\begin{prop} \label{persistent discontinuities}
Let $f \in \ACI$ and $p \in BP_{0}(f)$, then for any $\delta > 0$ there exists $\epsilon > 0$ such that if $g \in B(f,\epsilon)$ then $B(p,\delta) \cap BP_{0}(g) \neq \emptyset$.
\end{prop}

\begin{proof}

We can assume $\delta < \delta_{f}(p)$ (defined in \ref{Notation BP_0 and delta}).
Let us state some auxiliary notation
\begin{itemize}
\item $q_{-} = \displaystyle\lim_{x \to p^{-}} f(x)$, $q_{+} = \displaystyle\lim_{x \to p{+}} f(x) = f(p)$. 

\item $q$ the center of $(q_{-},q_{+})$, in particular $(q_{-},q,q_{+})$ is an ordered 3-tuple.

\item \(\delta_{1} \in (0,\frac{\delta}{2})\), such that $B(q_{-},2\delta_{1})$, $B(q,2\delta_{1})$ and $B(q_{+},2\delta_{1})$ are disjoint. Note that if $x \in B(q_{-},2\delta_{1})$, $y \in B(q,2\delta_{1})$, $z \in B(q_{+},2\delta_{1})$ then \((x,y,z)\) is an ordered 3-tuple.
\end{itemize}

Let \(I \subset \S^{1}\) be a non trivial interval such that \(\mu(I) = \delta_{2} < \delta_{1}\) and \(f(I) \subset B(q,\delta_{1})\).

We can assume that \(f(B^{*}(p,\delta_{2})_{-}) \subset B(q_{-},\delta_{1})\) and  \(f(B^{*}(p,\delta_{2})_{+}) \subset B(q_{+},\delta_{1})\). 
Note that \(B^{*}(p,\delta_{2})_{-},$ $B^{*}(p,\delta_{2})_{+}, I\) are disjoint intervals. Also if \(x_{-} \in B^{*}(p,\delta_{2})_{-}$, $x_{+} \in B^{*}(p,\delta_{2})_{+}\), and \(x_{0}\in I\) we have that \(x_{0} \notin [x_{-},x_{+}]\).

Let \(n, \epsilon\) be such that \(\frac{1}{n} < \delta_{2}\) and \(n \epsilon  < \delta_{1}\). 

Let \(g \in B(f,\epsilon)\). Since \(\mu(U_{n}(f,g)) \leq \frac{1}{n}\) there exist \(x_{-}, x_{0}, x_{+} \in \S^{1}\) such that \(x_{-} \in B(p,\delta_{2})_{-}\), \(g(x_{-}) \in B(q_{-},2\delta_{1})\), \(x_{+} \in B(p,\delta_{2})_{+}\), \(g(x_{+}) \in B(q_{+},2\delta_{1})\), \(x_{0} \in I\), \(g(x_{0}) \in B(q,2\delta_{1})\).

Since \((g(x_{-}),g(x_{0}),g(x_{+}))\) is an ordered 3-tuple and  \((x_{-},x_{0},x_{+})\) is not, applying Remark \ref{continuity in ordered tuples} we have that \(g\) is not continuous in \([x_{-},x_{+}]\), in particular \(BP_{0}(g) \cap B(p,\delta_{1}) \neq \emptyset\).
\end{proof}

\begin{coro} \label{continuidad superior de sharp}
The function $\sharp : \ACI \to \N$ is continuous for the right order topology in $\N$.
\end{coro}

\begin{prop}
The group $\ACI$ is not a complete metric space.
\end{prop}

\begin{proof}
We will construct $\{f_{n}\} \subset \ACI$ a non-convergent Cauchy sequence. 

Let $f_{0} = id$, and $f_{n}$ by

%$f_{n+1}(x) = \left \{ \begin{matrix} x + \frac{\epsilon_{n}}{6} & \mbox{ if } x \in [\frac{\epsilon_{n}}{3},\frac{\epsilon_{n}}{3} + \frac{\epsilon_{n}}{6})\\
%
%  x - \frac{\epsilon_{n}}{6} & \mbox{ if } x \in [\frac{\epsilon_{n}}{3} + \frac{\epsilon_{n}}{6},\frac{2\epsilon_{n}}{3})
%\\ 
%f_{n}(x) & \mbox{ if } x \notin [\frac{\epsilon_{n}}{3}, \frac{2\epsilon_{n}}{3}) \end{matrix}\right. $
%
%and $\epsilon_{n+1} = \min\{\frac{\epsilon_{n}}{3^{n+1}},\frac{\zeta_{n+1}}{3^{n}} \}$ where $\zeta_{n+1}$ is the value given by the Proposition \ref{persistent discontinuities} for $f = f_{n+1}$, $p = \frac{\epsilon_{n}}{2}$, and $\delta = \frac{\epsilon_{n}}{12}$.

%so $g \in B(f_{n},\zeta_{n})$, then $BP_{0}(g) \cap (\frac{\epsilon_{n}}{2} - \frac{\epsilon_{n}}{12},\frac{\epsilon_{n}}{2} + \frac{\epsilon_{n}}{12} ) \neq \emptyset$ for all $n > 0$.
%
%Since the intervals $(\frac{\epsilon_{n}}{2} - \frac{\epsilon_{n}}{12},\frac{\epsilon_{n}}{2} + \frac{\epsilon_{n}}{12} )$ are disjoints we conclude $\sharp (g) = \infty$ where we come to a contradiction.

$f_{n}(x) = \left \{ \begin{matrix} x & \mbox{ if } x \in [0,\frac{1}{2^{n}})\\

x + \frac{1}{2^{k+1}} & \mbox{ if } x \in [\frac{1}{2^{k}},\frac{1}{2^{k}} + \frac{1}{2^{k+1}}) \text{ for } 0 < k \leq n
\\ 

x - \frac{1}{2^{k+1}} & \mbox{ if } x \in [\frac{1}{2^{k}} + \frac{1}{2^{k+1}}, \frac{1}{2^{k-1}}) \text{ for } 0 < k \leq n. \end{matrix}\right. $

Note that $d(f_{n},f_{n+1}) \leq \frac{1}{2^{n}}$ then $f_{n}$ is a Cauchy sequence. Also, if $m > n$ then $d(f_{n},f_{m}) \leq \frac{1}{2^{n-1}}$. 

Suppose that $f_{n} \to g \in \ACI$.
Since \(f_{n}(x) = f_{m}(x)\) if \(m > n\) and \(x \geq \frac{1}{2^{n}}\). Then \(g(x) = f_{n}(x)\) for all \(x \geq \frac{1}{2^{n}}\).

We can conclude that \(p_{n} = \frac{1}{2^{n}} + \frac{1}{2^{n+1}} \in BP_{0}(g)\), so we get a contradiction.

%\(d(f_{m},g) \geq \epsilon > 0\) for all \(m > n\), where we come to a contradiction.
%
%, in this case \(d(g,f_{n}) \leq \frac{1}{2^{n+1}}\).
%
%Claim: \(p_{n} = \frac{1}{2^{n}} + \frac{1}{2^{n+1}} \in BP_{0}(g)\)
%
%If the claim is not true, then exists \(I\) a non trivial interval of \(p_{n}\) such that
%\(d(f_{n},g) \geq \int_{I} \d(f_{n}(x),g(x)) = \epsilon > 0\).

\end{proof}

In the following \(\rho:\R \to \ACI\) is a continuous homomorphism. Before the proof of the main theorem we will show some basic properties of \(\rho\) and the composition \(\sharp \circ \rho\).

Since $\R = \cup_{n \in \N}(-n a, na)$ for any $a > 0$, if a non empty open interval of $0$, $J$, verifies that $\rho(J) \subset\AC_{+}(\S^{1})$ then \(\rho(\R) \subset \AC(\S^{1})\). This result can be extended to sequences.

\begin{prop} \label{continuidad en una sucesion implica continuidad}
Let \(t_{n} \in \R^{+} \) be a sequence such that \(t_{n} \to 0\). If \(\rho(t_{n})\) is continuous for all \(n \in \N\) then \(\rho(\R) \subset \mathcal{AC}(\S^{1})\).
\end{prop}

\begin{proof}
Note that \(\rho(t) \in \mathcal{AC}(\S^{1})\) if and only if \(\rho(-t) \in  \mathcal{AC}(\S^{1})\). Then it is enough to prove that \(\rho(t) \in \mathcal{AC}(\S^{1})\) for \(t \in \R^{+}\).

Let \(t \in \R^{+}\), then exists a sequence \(n_{k}\) such that \(t = \sum_{k \in \N} t_{n_{k}}\).

Since \(\rho\) is continuous the sequence \(\rho(\sum_{k=1}^{N} t_{n_{k}}) \to \rho(t)\), i.e. there exists a sequence \(f_{N} = \rho(\sum_{k= 1}^{N} t_{n_{k}})\) of continuous functions such that \(f_{N} \to \rho(t)\). Finally, since the map \(\sharp\) is continuous for the right order topology we can conclude that \(\rho(t) \in \mathcal{AC}_{+}(\S^{1})\).
\end{proof}

The previous result is valid for domains.

\begin{prop} \label{continuidad en una sucesion implica continuidad en dominios}
Let \(D = \cup_{i} S_{\lambda_{i}}\) a domain, and \(\varrho: \mathbb{R} \to \mathcal{PAC}_{+}(D)\) a continuous homomorphism. Let \(t_{n} \in \R^{+} \) a sequence such that \(t_{n} \to 0\). If \(\varrho(t_{n})\) is continuous for all \(n \in \N\) then \(\varrho(\R) \subset \mathcal{AC}(D)\).
\end{prop}

\begin{prop} \label{sharp finite close to 0}
Let $\rho: \R \to \ACI$. There exists an open interval $I \subset \R$ with $0 \in I$ such that $(\sharp \circ \rho)(I)$ is bounded. 
\end{prop}

\begin{proof}
Suppose that there exists a sequence $t_{n} \in\R$ such that $t_{n} \to 0$, and $f_{n} = \rho(t_{n})$ verifies $\sharp(f_{n}) \to \infty$.

We can assume $t_{n} > 0$ and $\sum_{n} t_{n} < \infty$, we denote $r_{n} = \sum_{j=1}^{n} t_{j}$. 

The sequence defined by $\rho(r_{n})$ is convergent, since $r_{n}$ is a convergent sequence in $\R$ and the map $\rho$ is continuous. This result is valid for any subsequence of $t_{n}$.

We construct a sequence $s_{k} = \sum_{i=1}^{k} t_{n_{i}}$ such that the limit $s = \displaystyle\lim_{n\to +\infty} s_{n}$ verifies $\sharp (\rho(s)) = \infty$, so we get a contradiction.

Now we choose $t_{n_{k}}$ by induction.

\begin{itemize}
\item Let $t_{n_{1}}$, be such that $\sharp (f_{n_{1}}) > 2$.

\item For $h_{k} = \rho(s_{k})$ applying Corollary \ref{continuidad superior de sharp} exists $\epsilon_{k}$  such that $\sharp g \geq \sharp h_{k}$ for any $g \in B(h_{k},\epsilon_{k})$. This property can be translated to $\R$ by $\rho$, more precisely, there exists $\delta_{k}$ such that $\sharp (\rho(t)) \geq \sharp (\rho(s_{k}))$, for any $t \in (s_{k}-\delta_{k},s_{k} + \delta_{k})$.

\item To choose $t_{n_{k+1}}$ we will have two considerations.

\begin{enumerate}
\item $\sharp(\rho(t_{n_{k+1}})) \geq 3\sharp(\rho(s_{k}))$

\item $2^{k+1}t_{n_{k+1}} < \min\{\delta_{m} : m \leq k\}$
\end{enumerate}
\end{itemize}

Let us denote $s = \displaystyle\lim_{k \to +\infty} s_{k}$. By construction $\vert s - s_{k} \vert < \frac{\delta_{k}}{2^{k}}$, so $\sharp (\rho(s)) \geq \sharp (\rho(s_{k}))$.
In the other hand applying Remark \ref{Discontinuidad de la inversa} we have that $\sharp (\rho(s_{k + 1})) \geq 2 \sharp (\rho(s_{k})) \geq 2^{k + 1}$, so $\sharp ( \rho(s)) = \infty$.
\end{proof}

By composition this result can be translated to any interval $J \subset \R$, in particular $\sharp \circ \rho$ is locally bounded. 

Also a corollary of the previous proposition is the existence of maximum local for the function $\sharp \circ \rho$.

\begin{nota}
We denote by $T$ the set $T = \{ t \in \R : \sharp \circ \rho \mbox{ has a local max in } t\}$.

Note that if \(t \in T\) then \(-t \in T\). Also \(T\) is dense in \(\R\).
\end{nota}

Finally, by Proposition \ref{persistent discontinuities} we have continuity of $BP_{0}$ in a local maximum of $\sharp \circ \rho$.

\begin{obs} \label{continuity in local max}
Let $\rho: \R \to \ACI$ a continuous homomorphism, $t \in T$ and $M = \sharp \circ \rho(t)$. 

Then there exists an open interval \(J_{t}\) of $t$ such that $\sharp \circ \rho$ is constant in $J_{t}$, in particular $J_{t} \subset T$.

The function $Disc_{J_{t}}:J_{t} \to \left(\S^{1}\right)^{M}/ Perm(M)$ defined by $Disc(s) = BP_{0}(\rho(s))$ is continuous. 

\end{obs}

\begin{coro}
The set $T$ is open.
\end{coro}

\section{Proof of the main theorem}

Let \(\rho:\R \to \ACI\) is a continuous homomorphism.

In the case of $\rho(\R) \subset \AC_{+}(\S^{1})$, the thesis is verified immediately, so we can assume that this is not the case.

In the next we introduce the sketch of the proof main theorem. In particular the construction of the conjugation map \(f\).

If \(f \circ \rho(t) \circ f^{-1}\) is continuous and \(x \in BP_{0}(f)\) then applying Remark \ref{Discontinuidad de la inversa} we have that \(x \in \rho(t) \circ f^{-1}(BP_{0}((\rho(t) \circ f^{-1})))\). Also, by Remark \ref{Discontinuidad de la inversa} we have that  \(BP_{0}(\rho(t) \circ f^{-1}) \subset BP_{0}(f^{-1}) \cup f(BP_{0}(\rho(t))\), then \(x \in \rho(t)(BP_{0}(f)) \cup BP_{0}(\rho(-t))\). 
In other words, if \(x \notin \rho(t)(BP_{0}(f))$ then \(x \in BP_{0}(\rho(-t))\).  

We focus in the condition \(x \in BP_{0}(\rho(-t))\). In particular we define the ``cut'' points of \(f\) as the limit points of \(BP_{0}(\rho(t))\) with \(t \to 0\). We denote this set as \(B\).

To see that the set \(B\) is finite we prove \(B \subset BP_{0}(\rho(t))\) for \(t \in T\). We denote \(B_i = [b_i,b_{i+1})\) the intervals delimited by elements of \(B\).

We prove that for any \(B_i = [b_i,b_{i+1}) \) the function \(\rho(t)\) has at most three intervals of continuity if \(t\) close to \(0\). 
One of them uniformly big, and the others uniformly small, and with one endpoint in common with \(B_i\). To define the ``glue'' of \(f\), we study the dynamics of a sequence \(\rho(t_{k})\) with \(t_k \in T\) and \(t_k \to 0\) in the small intervals.

We finally conclude that the conjugation by \(f\) is continuous in this sequence and, by Proposition  \ref{continuidad en una sucesion implica continuidad en dominios}, we have the result for all \(t \in \R\).

\subsection{Construction of the singular set B (cut points)}

\begin{defi} \label{disc permanentes}
We denote by $B$ the set $$B = \{ b \in \S^{1} : \forall \delta_{1}, \delta_{2}^{\prime} > 0, \exists s \in (-\delta_{2}^{\prime},\delta_{2}^{\prime}) \mbox{ such that } BP_{0}(\rho(s)) \cap B(b,\delta_{1})\neq \emptyset \}.$$
\end{defi}

By compactness of \(\S^{1}\) we have the following remark on \(B\).

\begin{obs} \label{disc cerca de B} \

\begin{enumerate}
\item The set $B$ is not empty.

\item Let $b \in \S^{1}$, then $b \in B$ if and only if there exist a pair of sequence  $t_{n} \in \R$ and $x_{n} \in \S^{1}$ such that $t_{n} \to 0$, $x_{n} \in BP_{0}(\rho(t_{n}))$ and $x_{n} \to b$.

\item Let $t_{n} \in \R$ and $x_{n} \in \S^{1}$ be sequences such that $t_{n} \to 0$, $x_{n} \in BP_{0}(\rho(t_{n}))$. Then any limit point of $x_{n}$ is in $B$.

\item For any $\delta_{1} > 0$ there exists $\delta_{2}^{\prime} > 0$ such that $BP_{0}(\rho(t)) \subset B(B,\delta_{1})$ for all $t \in (-\delta_{2}^{\prime},\delta_{2}^{\prime})$.
\end{enumerate}
\end{obs}

The item 2 of the previous Remark can be translated to any \(t \in \R\).

\begin{prop} \label{densitiy of B in Im(rho)}
Let $t \in \R$ and $b \in B$. If $b \notin BP_{0}(\rho(t))$ then for any pair $\delta, \delta_{0}^{\prime} \in  \R^{+}$, there exists $s \in (t - \delta_{0}^{\prime}, t + \delta_{0}^{\prime})$ such that $BP_{0}(\rho(s)) \cap B(b,\delta) \neq \emptyset$.
\end{prop}

\begin{proof}

We can assume that $\delta < \frac{1}{4}$.

Let $b \in B$ be such that $\delta_{1} = \d(b,BP_{0}(\rho(t))) > 0$. If $\delta_{1} < \delta$ then $s = t$ verifies the thesis, so we can assume that $\delta_{1} \geq \delta$. By definition of $B$ there exists $\delta_{2}^{\prime} > 0$ such that for any $\delta_{3}^{\prime} \leq \delta_{2}^{\prime}$ exists $t_{1} \in (-\delta_{3}^{\prime},\delta_{3}^{\prime})$ with $d(b,BP_{0}(\rho(t_{1}))) < \frac{\delta}{8}$. We can take \(\delta_{2}^{\prime} \leq \delta_{0}^{\prime}\).

Let $N_{1} \in \N$ such that $\frac{1}{N_{1}} < \frac{\delta}{8}$. Since $\rho$ is continuous, we can take $t_{1} < \delta_{2}^{\prime}$ such that $d(\rho(t_{1}),\rho(0)) = d(\rho(t_{1}),id) < \left(\frac{1}{N}\right)\left(\frac{\delta}{8}\right) < \left(\frac{\delta}{8}\right)^{2}$.

Let $x_{0} \in BP_{0}(\rho(t_{1})) \cap B(b,\frac{\delta}{8})$, $x_{-} \in B(b,\frac{\delta}{2})_{-} \setminus \left[ B(b,\frac{\delta}{4}) \cup U_{N_{1}}(\rho(t_{1}),id) \right]$
and $x_{+} \in B(b,\frac{\delta}{2})_{+} \setminus \left[ B(b,\frac{\delta}{4}) \cup U_{N_{1}}(\rho(t_{1}),id)\right]$. We also denote \(y_{-} = \rho(t_{1})(x_{-})\) and \(y_{+} = \rho(t_{1})(x_{+})\). 

By construction $\d(x_{-},y_{-}) \leq N_{1}d(\rho(t_{1}),id) < \frac{\delta}{8}$,  and also $\d(x_{+},y_{+}) < \frac{\delta}{8}$. In particular $y_{-}, y_{+} \in B(b,\delta)\). Since $(x_{-},b,x_{+})$ is an ordered 3-tuple, $\d(x_{-},b) \geq \frac{\delta}{4}$ and $\d(x_{+},b) \geq \frac{\delta}{4}$ then $(y_{-},b,y_{+})$ is an ordered 3-tuple and $[y_{-},y_{+}] \subset B(b,\delta) \subset Cont(\rho(t))$. We conclude that \(\rho(t)\) restricted to the interval \([y_{-},y_{+}]\) is a homeomorphism into its image. We now consider \(\rho(t_{1})([x_{-},x_{+}])\).

If $\rho(t_{1})([x_{-},x_{+}]) \subset B(b,\delta) \subset Cont(\rho(t))$, then $x_{0} \in BP_{0}(\rho(t) \circ \rho(t_{1})) = BP_{0}(\rho(t + t_{1}))$, so $s = t + t_{1}$ verifies the thesis.

In the other hand, there exists $x^{\prime} \in [x_{-},x_{+}]$ such that $y^{\prime} = \rho(t_{1})(x^{\prime}) \notin B(b,\delta) \subset Cont(\rho(t))$. 
Since $\rho(t)\vert_{[y_{-},y_{+}]}: [y_{-},y_{+}] \to [\rho(t)(y_{-}),\rho(t)(y_{+})]$ is a homeomorphism then $(\rho(t)(y_{-}) , \rho(t) (y^{\prime}), \rho(t)(y_{+}) )$ isn't an ordered 3-tuple. 
Then $s = t + t_{1}$ verifies that $\rho(s)$ has a discontinuity point in the interval $[x_{-},x_{+}] \subset B(b,\frac{\delta}{2})$.
\end{proof}

\begin{prop}
Let $t \in T$, then $B \subset BP_{0}(\rho(t))$.
\end{prop}

\begin{proof}
Suppose the thesis is not true. Take $b \in B$ such that $\delta_{1} = \d(b,BP_{0}(t)) > 0$. By Remark \ref{continuity in local max} there exists $\delta_{2}^{\prime} > 0$ such that any $s \in (t - \delta_{2}^{\prime},t + \delta_{2}^{\prime})$ verifies $\d(b,BP_{0}(s)) > \frac{\delta_{1}}{2}$.

Also applying Proposition \ref{densitiy of B in Im(rho)} to $t$ with parameters $b, \frac{\delta_{1}}{2}, \delta_{2}^{\prime}$ there exists $s \in (t - \delta_{2}^{\prime}, t + \delta_{2}^{\prime})$ such that $BP_{0}(\rho(s))  \cap B(b,\frac{\delta_{1}}{2}) \neq \emptyset$, so we get a contradiction.
\end{proof}

\begin{coro}\label{B finite}\

\begin{enumerate}
\item The set $B$ (Definition \ref{disc permanentes}) is finite and non empty.

\item For any $t \in T$, there exists $I$ an interval of $t$ such that $B \subset BP_{0}(\rho(s)),\) for any \(s \in I$.
\end{enumerate}
\end{coro}

\begin{nota}  \label{Notation B_i}
We use the following notation $N = \# B > 0$ and $B = \{b_{1},...,b_{N}\}$ with the same criteria of the set \(BP_{0}\) (Definition \ref{Notation BP_0 and delta}). Let $B_{i} = [b_{i},b_{i+1})$ for $i < N$ and $B_{N} = [b_{N},b_{1})$. 

Finally, we denote $\delta_{B} = \min\{\d(b_{i},b_{j}) : i \neq j\}$.
\end{nota}

\begin{obs}
Let $t \in T$. Then for all $i \leq \sharp(\rho(t))$, there exists $j \leq N$ such that $I_{\rho(t)}(i) \subset B_{j}$.
\end{obs}

\subsection{Properties of the singular set B}

Now we study the homomorphism $\rho$ restricted to intervals of 0. In particular the properties of the set $BP_{0}(\rho(t))$ when $\sharp \circ \rho$ has a relative maximum in $t$ (i.e.: \(t \in T\)).

\begin{prop} \label{discontinuidades que no estan en B van a B}
There exists $I$ an interval of 0 such that for any $t \in I \cap T$ if $p \in BP_{0}(\rho(t))\setminus B$ then 
\begin{enumerate}
\item $\rho(t)(p) \in B$,
\item $\lim_{x \to p^{-}} \rho(t)(x) \in B$.
\end{enumerate}
\end{prop}

\begin{proof}
Let us prove item 1. 

Suppose the thesis is not true, then there exist  sequences $\{t_{n}\} \subset T$ and $x_{n} \in BP_{0}(\rho(t_{n})) \setminus B$ such that $t_{n} \to 0$ and $\rho(t_{n})(x_{n}) \not \in B$. Since $t_{n} \to 0$ applying item 3 of Remark \ref{disc cerca de B} we have that $\d(x_{n},B) \to 0$, then we can assume $x_{n} \to a \in B$. On the other hand $\rho(t_{n})(x_{n}) \in BP_{0}(\rho(-t_{n}))$, then we can assume $\rho(t_{n})(x_{n}) \to b \in B$.

Let $t \in T$. By definition $B(p,\frac{\delta_{\rho(t)}}{2})  \cap B(q,\frac{\delta_{\rho(t)}}{2}) = \emptyset$ for any pair $p \neq q$, $p,q \in BP_{0}(\rho(t))$. Since the function $Disc$ defined in Remark \ref{continuity in local max} is continuous in $t$, there exists $\delta_{1}^{\prime} \in \R^{+}$ such that $\#(BP_{0}(\rho(s)) \cap B(x,\frac{\delta_{\rho(t)}}{2})) = 1$ for any $s \in (t-\delta_{1}^{\prime},t+\delta_{1}^{\prime})$ and $x \in BP_{0}(\rho(t))$.
Also, if \(\delta_{1}^{\prime}\) is small enough $(t-\delta_{1}^{\prime},t+\delta_{1}^{\prime}) \subset T$ then $B \subset BP_{0}(\rho(s))$ for all $s \in (t-\delta_{1}^{\prime},t+\delta_{1}^{\prime})$, in particular $a, b \in BP_{0}(\rho(s))$.

Let $n$ be such that $t + t_{n} \in (t - \delta_{1}^{\prime} , t + \delta_{1}^{\prime})$, $0 < \d(x_{n},a) < \frac{\delta_{\rho(t)}}{2}$ and $0 < \d(\rho(t_{n})(x_{n}),b) < \frac{\delta_{\rho(t)}}{2}$.
Since $x_{n} \in BP_{0}(\rho(t_{n}))$ and $\rho(t_{n})(x_{n}) \in B(b, \frac{\delta_{\rho(t)}}{2}) \setminus \{b\} \subset Cont(\rho(t))$, then $x_{n} \in BP_{0}(\rho(t + t_{n}))$.
We conclude  $\#BP_{0}(\rho(t + t_{n})) \cap B(a,\frac{\delta_{\rho(t)}}{2}) \geq 2$ where we have a contradiction.

Let us prove item 2.

Suppose the thesis is not true, then there exist sequences $\{t_{n}\} \subset T$, $t_{n} \to 0$ and $x_{n} \in BP_{0}(\rho(t_{n})) \setminus B$ such that $y_{n} = \displaystyle\lim_{x \to x_{n}^{-}} \rho(t_{n})(x) \not \in B$. We can assume $x_{n} \to a \in B$. By Remark \ref{left limit in BP part 1} we have that $y_{n} \in BP_{0}(\rho(-t_{n}))$, so we can assume $y_{n} \to c \in B$.

Let $\delta_{1}^{\prime} \in \R^{+}$ be as in item 1 and \(n \in \N\) such that $t + t_{n} \in (t-\delta_{1}^{\prime},t+\delta_{1}^{\prime})$, $0 < \d(x_{n},a) < \frac{\delta_{\rho(t)}}{4}$ and $0 < \d(y_{n},c) < \frac{\delta_{\rho(t)}}{4}$.

Note that $x_{n} \in Cont(\rho(t + t_{n}))$, in other words
$$\lim_{x \to x_{n}^{-}} \rho(t + t_{n})(x) = \lim_{x \to x_{n}^{+}} \rho(t + t_{n})(x),$$
also since  $y_{n} \in Cont(\rho(t))$ we have that
$$\lim_{x \to x_{n}^{-}} \rho(t + t_{n})(x)= \lim_{y \to y_{n}^{-}} \rho(t)(y) = \rho(t)(y_{n}).$$

Finally, since $\rho(t)(y_{n}) \in Cont(\rho(-t))$ we have
$$\lim_{x \to x_{n}^{-}} \rho(t_{n})(x) = \lim_{x \to x_{n}^{-}} \rho(-t) \circ \rho(t + t_{n})(x) = \lim_{y \to y_{n}} \rho(-t) \circ \rho(t)(y) = y_{n}$$
and
$$ \lim_{x \to x_{n}^{+}} \rho(t_{n})(x) = \lim_{x \to x_{n}^{+}} \rho(-t) \circ \rho(t + t_{n})(x) = \lim_{y \to y_{n}} \rho(-t) \circ \rho(t)(y) = y_{n}$$

where we have a contradiction.
\end{proof}

\begin{coro}
Let \(I \subset \R\) be the interval of the previous Proposition.
The function $(\sharp \circ \rho)\vert_{I}$ is bounded. Also $(\sharp \circ \rho)\vert_{I} \leq 2\# B$.
\end{coro}

By item 4 of Remark \ref{disc cerca de B} and item 1 of Corolary \ref{B finite} for $t$ close to 0 the intervals $I_{\rho(t)}$ can be divided into two classes, the small and the big ones.

Let \(I\) be an interval of \(0\) such that for any \(t \in I\) verifies \(BP_{0}(\rho(t)) \subset B(B,\frac{\delta_{B}}{8})\).
 
\begin{defi}
For any $t \in I$ we say that an interval $I_{\rho(t)}$ is of:
\begin{itemize}
\item Type 1 if $\mu(I_{\rho(t)}) \leq \frac{\delta_{B}}{4}$.
\item Type 2 if $\mu (I_{\rho(t)}) \geq \frac{\delta_{B}}{2}$.
\end{itemize}
\end{defi}

\subsection{Dynamics of intervals of type 1 and type 2}

%In the examples \ref{torus action main teo} and \ref{conjugation example main teo}, $\rho(t)$ has at most one discontinuity between two point of $B$. This result is not general in our context (Example \ref{ejemplo con dos discontinuidades en un intervalo}).

We are going to show that for \(t\) close to \(0\) and $b \in B$ we have $\# (BP_{0}(\rho(t)) \cap B^{*}(b,\frac{\delta_{B}}{2})_{+}) \leq 1$ and $\# (BP_{0}(\rho(t)) \cap B^{*}(b,\frac{\delta_{B}}{2})_{-}) \leq 1$. 
In particular \(\#  (BP_{0}(\rho(t)) \cap (B_{i}\setminus B)) \leq 2\).

\begin{prop}
There exists $J$ an interval of 0 such that $\# (BP_{0}(\rho(t)) \cap B^{*}(b,\frac{\delta_{B}}{2})_{+}) \leq 1$ for all $t \in J \cap T$ and $b \in B$.
Also $\# (BP_{0}(\rho(t)) \cap B^{*}(b,\frac{\delta_{B}}{2})_{-}) \leq 1$.
\end{prop}

\begin{proof}
We prove the thesis in $B^{*}(b,\frac{\delta_{B}}{2})_{+}$, the case $B^{*}(b,\frac{\delta_{B}}{2})_{-}$ is analogous.

Suppose the thesis is not true, then there exist sequences $\{x_{n}\},\{ y_{n}\}$ of $\S^{1}$ and $t_{n} \in T$ such that $t_{n} \to 0$, $x_{n} \neq y_{n}$, $x_{n},y_{n} \in BP_{0}(\rho(t_{n})) \cap B^{*}(b,\frac{\delta_{B}}{2})_{+}$. Then by Remark \ref{disc cerca de B} we have that $x_{n},y_{n} \to b \in B$. We can assume $(b,x_{n},y_{n})$ is an ordered 3-tuple and $\rho(t_{n})$ is continuous in $[x_{n},y_{n})$, i.e. $[x_{n},y_{n}) = I_{\rho(t_{n})}(i_{n})$ for some $i_{n}$. 

Let \(I\) be an interval such that verifies Proposition \ref{discontinuidades que no estan en B van a B}, if $t_{n} \in I \cap T$ we have $\rho(t_{n})[x_{n},y_{n}) = [a_{n},c_{n})$ with $a_{n},c_{n} \in B$.

The sequence $d_{n} =  \mu(\rho(-t_{n})([a_ {n},c_{n})))$ converges to 0, but $ \mu(a_ {n},c_{n}) \geq \delta_{B} > 0$. Since $\rho(-t_{n}) \to \rho(0) = id$ which contradicts the Proposition \ref{function close to identity chenge little the measure}. 
\end{proof}

\begin{obs} \label{type 1 intersection B = 1} 
If $I_{\rho(t)}(i)$ is an interval of type 1, then $\overline{I_{\rho(t)}(i)} \cap B \neq \emptyset$. Moreover $\# ( \overline{I_{\rho(t)}(i)} \cap B) = 1$.

In the Example \ref{ejemplo con dos discontinuidades en un intervalo} we show that this is not true for intervals of type 2,
\end{obs}

For \(t\) close to 0 we have that \(U(id,\rho(t))[\frac{\delta_{B}}{8}] \leq \frac{\delta_{B}}{8}\) (Proposition \ref{convergence by sets}), we conclude that  \(\rho(t)(I_{\rho(t)}(i)) \cap I_{\rho(t)}(i) \neq \emptyset\) if \(I_{\rho(t)}(i)\) is an interval of type 2.
Since $\rho(t)(I_{\rho(t)}(i)) = I_{\rho(-t)}(j)$ we can conclude the following result.

\begin{obs} \label{type 2 no change}
There exists $I$ an interval of 0 such that, for any $t \in J \cap T$, it holds,
\begin{enumerate}
\item If $I_{\rho(t)}(i)$ is an interval of type 1 for $\rho(t)$, then exists $j$ such that $I_{\rho(-t)}(j) = \rho(t)(I_{\rho(t)}(i))$ is an interval of type 1 of $\rho(-t)$.
\item 
If $I_{\rho(t)}(i)$ is an interval of type 2 and $I_{\rho(t)}(i) \subset B_{k}$ then $\rho(t)(I_{\rho(t)}(i)) \subset B_{k}$.
\end{enumerate}
\end{obs}

To construct the domain $D$ of the main theorem, we have to cut and glue the intervals $B_{j}$ in a  ``good order'' and the map $f$  fulfills that role.
More precisely, \(f^{-1}\) of each connected component of \(D\) will be $\tilde{C_{j}} =  \cup_{k \in K(j)} B_{k}$, where $K(j)$ is a subset of $\{1,...,N\}$.

Any interval \([x_{1},x_{2})\) of type 1 of \(\rho(t)\) has exactly one of this endpoints in \(B\).
Since \([y_{1},y_{2}) = \rho(t)([x_{1},x_{2}))\) is an interval of type 1 for \(\rho(-t)\), applying Proposition \ref{discontinuidades que no estan en B van a B} and Remark \ref{type 1 intersection B = 1} we have that \(x_{1} \in B\) if and only if \(y_{1} \notin B\). This result can be translated to semi intervals of \(B_{i}\).

\begin{obs} \label{intervalos chicos cambian de lado}
If $I_{\rho(t)}(i)$ is an interval of type 1 and $I_{\rho(t)}(i)  \subset (B_{j})_{+}$, then  $\rho(t)(I_{\rho(t)}(i)) \subset (B_{j_{1}})_{-}$ for some $j_{1}$. Analogous if $I_{\rho(t)}(i)$ is an interval of type 1 and $I_{\rho(t)}(i)  \subset (B_{j})_{-}$, then  $\rho(t)(I_{\rho(t)}(i)) \subset (B_{j_{2}})_{+}$ for some $j_{2}$.
\end{obs}

\begin{prop} \label{t o -t} There exists $I \subset \R$ interval of 0, such that any $t \in I \cap T$ verifies,

\begin{enumerate}
\item if $I _{\rho(t)}(i) \subset (B_{j})_{-}$ is an interval of type 1 of $\rho(t)$ then $\rho(-t)$ has not an interval of type 1 in $(B_{j})_{-}$.

\item if $I _{\rho(t)}(i) \subset (B_{j})_{+}$ is an interval of type 1 of $\rho(t)$ then $\rho(-t)$ has not an interval of type 1 in $(B_{j})_{+}$.
\end{enumerate}
\end{prop}

\begin{proof}
We will prove the item 1, the item 2 is analogous.

Let $[b_{i},a)$ an interval of type 1 for $\rho(t)$ then there exists $c$ such that $[a,c) \subset B_{i}$ is an interval of type 2 for $\rho(t)$. Applying Proposition \ref{discontinuidades que no estan en B van a B} item 1 and Remark \ref{type 2 no change} item 2 we have $\rho(t)(a) = b_{i}$.

Suppose that $[b_{i},a^{\prime})$ is an interval of type 1 for $\rho(-t)$. Applying Proposition \ref{intervalos chicos cambian de lado} there exists $k$ such that $\rho(-t)(b_{i}) \subset (B_{k})_{+}$ but $\rho(-t)(b_{i}) = a \in (B_{i})_{-}$. This is a contradiction.
\end{proof}

Since $B$ is finite we can define the following functions.

\begin{defi}
Let $\sigma_{1}, \sigma_{2}, \tau_{1},\tau_{2} : \{1,...,N\} \to \{1,..,N\}$  be functions and  $\{t_{q}\} \subset \R^{+} \cap T$ be a sequence and such that
\begin{enumerate}
\item \(t_{q} \to 0\)
\item  $\rho(t_{q})(b_{i}) \in B_{\sigma_{1}(i)}$ and $\lim_{x \to (b_{i+1})^{-}}\rho(t_{q})(x) \in B_{\tau_{1}(i)}$ for all $q$,
\item  $\rho(-t_{q})(b_{i}) \in B_{\sigma_{2}(i)}$ and $\lim_{x \to (b_{i+1})^{-}}\rho(-t_{q})(x) \in B_{\tau_{2}(i)}$ for all $q$.
\end{enumerate}
\end{defi}

Note that $\sigma_{1}(i) \neq i$ if and only if there exists \(c \in (B_{i})_{-}\) such that \([b_{i},c) \subset (B_{i})_{-}\) is an interval of type 1. Then applying Proposition \ref{t o -t} we have that if $\sigma_{1}(i) \neq i$ then $\sigma_{2}(i) = i$ also if $\sigma_{2}(i) \neq i$ then $\sigma_{1}(i) = i$. The functions $\tau_{i}$ verifies the same properties.

\begin{defi}
We define $\tau, \sigma:\{1,...,N\} \to \{1,...,N\}$ by

$\tau(i) = \left \{ \begin{matrix} \tau_{1}(i) & \mbox{if } \tau_{1}(i) \neq i
\\ \tau_{2}(i) & \mbox{if } \tau_{2}(i) \neq i
\\ i & \mbox{if } \tau_{1}(i) = \tau_{2}(i) = i
\end{matrix}\right. ,
\hspace{1cm}
\sigma(i) = \left \{ \begin{matrix} \sigma_{1}(i) & \mbox{if } \sigma_{1}(i) \neq i
\\ \sigma_{2}(i) & \mbox{if } \sigma_{2}(i) \neq i
\\ i & \mbox{if } \sigma_{1}(i) = \sigma_{2}(i) = i
\end{matrix}\right.$
\end{defi}

\begin{obs} \label{pseudoinversa de salto}\

\begin{enumerate}
\item If $\tau(i) \neq i$ then $\sigma (\tau (i)) = i$.

\item If $\sigma(i) \neq i$ then $\tau(\sigma(i)) = i$.
\end{enumerate}
\end{obs}

By Remark \ref{type 2 no change} we get the following result.

\begin{obs} \label{pre inavriant 1}
The functions $\tau$ and $\sigma$ verify that
\begin{enumerate}
\item $\rho(t_{q})(B_{i}) \subset B_{\tau_{1}(i)} \cup B_{i} \cup B_{\sigma_{1}(i)} \subset B_{\tau(i)} \cup B_{i} \cup B_{\sigma(i)}$,

\item  $\rho(-t_{q})(B_{i}) \subset B_{\tau_{2}(i)} \cup B_{i} \cup B_{\sigma_{2}(i)} \subset B_{\tau(i)} \cup B_{i} \cup B_{\sigma(i)}$.
\end{enumerate}
\end{obs}
%
%\textcolor{red}{Other properties we obtain about $\sigma$ and $\tau$ is related to the order of the intervals $B_{i}$}
%
%\begin{prop} \label{saltos reales} No se si lo preciso \\
%For any $i$ it is verify that
%\begin{enumerate}
%\item $\sigma(i) \neq i-1$
%\item $\tau(i) \neq i+1$
%\end{enumerate}
%\end{prop}
%
%\begin{proof}
%We prove 1, the proof of 2 is analogous.
%
%If $a \in (BP_{0}(\rho(t_{q})) \cap (B_{i})_{-}) \setminus \{b_{i}\}$ then applying \ref{type 2 no change} it is verify that $\rho(a) = b_{i}$. So if $\sigma(i) = i-1$ then by definition $\lim_{x \to a^{-}} \rho(t_{q})(a) = b_{i}$ and then $a$ is a continuity point of $\rho(t_{q})$ where we come to a contradiction.
%\end{proof}

%The case of \(\sigma(i) = i\) correspond to \(b_{i}\) is a global fix point (\(\rho(t)(b_{i}) = b_{i}\)). This be seen later, but now we starts to study this case.

Now we study the dynamics of \(\tau\) and \(\sigma\). 

\begin{prop} \label{Inicio de Ci}
Let $i$ be such that $\sigma(i) = i$ and $k_{0}$ the first index such that $\tau^{k_{0}}(i)\subset \{\tau^{k}(i) : k < k_{0}\}$. Then $\tau^{k_{0}}(i) = \tau^{k_{0}-1}(i)$.
\end{prop}

\begin{proof}
The existence of $k_{0}$ is guaranteed by the finitude of $B$. If $k_{0} = 1$ the result is immediate.

Suppose $k_{0} > 1$ and the thesis is not true.

By definition of $k_{0}$ we have $\tau^{h}(i) \neq \tau^{h+1}(i)$ for all $h < k_{0}-1$. Let $h < k_{0} - 1$ such that $\tau^{k_{0}}(i) = \tau^{h}(i)$.
 
If $h = 0$ (i.e. $\tau^{k_{0}}(i) = i$), applying Remark \ref{pseudoinversa de salto} we have the following equalities $i = \sigma(i) = \sigma(\tau^{k_{0}}(i)) = \tau^{k_{0} - 1}(i)$, this is a contradiction.
 
In the case $h>0$ by Remark \ref{pseudoinversa de salto} $ \tau^{k_{0}-1}(i) = \sigma(\tau^{k_{0}}(i)) = \sigma( \tau^{h}(i)) = \tau^{h-1}(i)$ which is a contradiction.
\end{proof}

Interchanging $\tau$ and $\sigma$, we have the symmetric proposition.
 
\begin{prop} \label{fin de Ci}
Let $i_{0}$ such that $\tau(i) = i$.
Let $k_{0}$ be the first index such that $\sigma^{k_{0}}(i)\subset \{\sigma^{k}(i) : k < k_{0}\}$. Then $\sigma^{k_{0}}(i) = \sigma^{k_{0}-1}(i)$.
\end{prop}

The previous Proposition allows us to separate $\{1,...,N\}$ in two categories of invariant set for $\tau$ and $\sigma$.

\begin{nota}
We denote

\begin{enumerate}
\item $O_{\tau}(i) = \{ \tau^{k}(i) : k \geq 0\}$ and $O_{\sigma}(i) = \{ \sigma^{k}(i) : k \geq 0\}$,
\item $O(i) = O_{\tau}(i) \cup O_{\sigma}(i)$ and $k_{i}= \#O(i)$.
\item $L=\{i: \sigma(i) = i \}$. Note that $O(l) = O_{\tau}(l)$ for \(l \in L\).
\item $S = \{1,...,N\} \setminus (\cup_{l \in L} O(l) )$
\item $O = \{O(i) : i \in L \cup S \}$ and $M = \# O$.
\end{enumerate}
\end{nota}

Let \(i \in \{1,...,N\}\) and \(k \in \Z^{+}\) such that \(\sigma^{k-1}(i) \neq \sigma^{k}(i) = \sigma^{k+1}(i)\) (i.e. \(\sigma^{k}(i) \in L\)). Applying Remark \ref{pseudoinversa de salto} we have that \(i = \tau^{k}(\sigma^{k}(i)) \in O_{\tau}(\sigma^{k}(i))\) in particular \(i \notin S\).

In the other hand if $i \in \{1,...,N\}$ verifies that $\tau^{k-1}(i) \neq \tau^{k}(i) = \tau^{k+1}(i)$ then applying Proposition \ref{fin de Ci} to \(\tau^{k}(i)\) there exists $h \geq k$ such that $ \sigma^{h-1}(\tau^{k}(i)) \neq \sigma^{h}(\tau^{k}(i)) = \sigma(\sigma^{h}(\tau^{k}(i))) \in L$.
 Applying Remark \ref{pseudoinversa de salto} we have that \(\sigma^{h}(\tau^{k}(i)) = \sigma^{h-k}(i)\) and \(i = \tau^{h-k}(\sigma^{h-k}(i))\) then $i \in O_{\tau}(\sigma^{h-k}(i))$ and $i \notin S$. 

In summary we have the following Remark.

\begin{obs} \label{pre cilcos en sigma y tau}
Let $i \in \{1,...,N\}$, $i \in S$ if and only if $\tau^{k+1}(i) \neq \tau^{k}(i)$ and $\sigma^{k+1}(i) \neq \sigma^{k}(i)$ for all $k$.
\end{obs}

\begin{prop}
Let $s \in S$. The functions $\sigma \vert_{O(s)}$ and $\tau \vert_{O(s)}$ are bijection with 1 orbit, in particular $O(s) = O_{\sigma}(s) = O_{\tau}(s)$.
\end{prop}

\begin{proof}
Let prove the case of $\sigma$.

Let $k$ the first index such that $\sigma^{k}(s) \in \{\sigma^{h}(s) : h < k\}$. Note that $k > 1$. Applying Remark \ref{pre cilcos en sigma y tau} we have $\sigma^{k}(s) \neq \sigma^{k-1}(s)$.

Let $h < k$ such that $\sigma^{h}(s) = \sigma^{k}(s)$. By Remark \ref{pseudoinversa de salto} we have $s = \sigma^{k-h}(s)$, then $h = 0$.

We conclude that the function $\sigma$  is conjugated to $x \to x + 1$  in $\Z_{k_{s}}$. In particular is a bijection with one orbit. Since $k_{s} > 1$ we conclude the same properties to $\tau$, even more $\tau \vert_{O(s)} = (\sigma \vert_{O(s)})^{-1}$. Finally  $O(s) = O_{\sigma}(s) = O_{\tau}(s)$.
\end{proof}

\begin{prop}
Let \(i_{1}, i_{2} \in L\), then \(i_{1} = i_{2}\) or \(O(i_{1}) \cap O(i_{2}) = \emptyset\).
\end{prop}
\begin{proof}
Supose that \(O(i_{1}) \cap O(i_{2}) \neq \emptyset\), i.e.: there exists \(h_{1}, h_{2}\) such that \(\tau^{h_{1}}(i_{1}) = \tau^{h_{2}}(i_{2})\).
Note that if \(h_{1} = h_{2} = 0\) then \(i_{1} = i_{2}\).
We can assume that \(h_{2} \neq 0\) and \(\tau^{h_{2}}(i_{2}) \neq \tau^{h_{2}-1}(i_{2})\).

If \(h_{1} = 0\), as \(i_{1} \in L\) we have that \(\sigma(\tau^{h_{2}}(i_{2})) = \sigma(i_{1})=i_{1} = \tau^{h_{2}}(i_{2})\). Since \(\tau^{h_{2}}(i_{2}) \neq \tau^{h_{2}-1}(i_{2})\) applying Remark \ref{pseudoinversa de salto} we get that \(\sigma(\tau^{h_{2}}(i_{2})) = \tau^{h_{2} - 1}(i_{2})\) which is a contradiction.

If \(h_{1} > 0\) we can assume that \(\tau^{h_{1}}(i_{1}) \neq \tau^{h_{1}-1}(i_{1})\) then 
applying Remark \ref{pseudoinversa de salto} we have that \(\tau^{h_{1}-1}(i_{1}) = \sigma(\tau^{h_{1}}(i_{1})) = \sigma(\tau^{h_{2}}(i_{2})) = \tau^{h_{2}-1}(i_{2})\). Finally, by induction we conclude  \(i_{1} = i_{2}\) and we get a contradiction.
\end{proof}

\begin{coro}
Let $i, j \in \{1,...,N\}$. Then $O(i)=O(j)$ or $O(i) \cap O(j) = \emptyset$.

Also, applying Remark \ref{pre inavriant 1} we have that for any $i \in \{1,...,N\}$ the set $\cup_{k \in O(i)} B_{k}$ is invariant for $\rho(t_{q})$ and $\rho(-t_{q})$ for all $q$.
\end{coro}

The sets $O(i)$ give a partition of the set $\{1,...,N\}$, and $\S^{1}$ via the interval $B_{i}$. For the sake of  clarity we choose a representative element of each $O(i)$ with some properties.

\begin{nota}
We will denote $(m_{0}, ... ,m_{M-1})$ a $M$ tuple such that $O = \{O_{\tau}(m_{i}) : i < M  \}$. Note that $O_{\tau}(m_{i}) \neq O_{\tau}(m_{j})$ if $i \neq j$.
%In particular for $m_{j}$ we have $O_{\tau}(m_{j}) = O(m_{j})$. 
\end{nota}

\begin{obs}
For any $x \in \S^{1}$ there exists an unique pair $(j,h)$ with $j < M$ and $h < k_{m_{j}}$ such that $x \in B_{\tau^{h}(m_{j})}$.
\end{obs}

\subsection{Construction of the map \(f\).}

Finally we will build $\S^{1}_{\lambda_j}$ the component of $\mathcal{D}$ and the function $f \in IET(\S^{1},
\mathcal{D})$ of the main theorem.

Recall the definition of maps \(\mathcal{I}_{j}:(\S^{1})_{\lambda_{j}} \to D\) (Definition \ref{inclusion in domain}) canonical inclusion of j-th component of a domain \(D\), and \(\pi_{\lambda_{j}}:\R \to (\S^{1})_{\lambda_{j}}\) (Notation \ref{distance S1}) the canonical projection.

\begin{defi}
For \(j \in \{0,..,M-1\}\) we define \(\lambda_{j} = \sum_{i < k_{j}}\mu(B_{\tau^{i}(m_{j})})\), and \(\mathcal{D} = \uplus_{j < M} (\S^{1})_{\lambda_{j}}\)

We define the function \(f \in IET(\S^{1},\mathcal{D})\) by

$$f(x) = \mathcal{I}_{j}\left(\pi_{\lambda_{j}}\left(diam([b_{\tau^{h}(m_{j})},x)) +  \sum_{l < h}\mu(B_{\tau^{l}(m_{j})})\right)\right)$$

where $h < k_{m_{j}}$ and $x \in B_{\tau^{h}(m_{j})}$.

\end{defi}

%\begin{defi}
%We define $c_{j}$ by $c_{j}= \frac{j}{M}$ and the set $C_{j}$ by  $C_{j} = \left[\frac{j}{M},\frac{j+1}{M}\right)$ for $j \in \{0,...,M-1\}$.
%
%We define the function $f^{-1}:[0,1) \to [0,1)$ define by $$f^{-1}(x) =  \frac{j}{M} + \frac{x - b_{\tau^{h}(m_{j})}}{\mu(B_{\tau^{h}(m_{j})})M k_{m_{j}}} + \frac{h}{Mk_{m_{j}}} \hspace{0.2cm} \mbox{ (mod 1) }$$
%
%where $h < k_{m_{j}}$ and $x \in B_{\tau^{h}(m_{j})}$.
%
%Dibujo?
%\end{defi}

\begin{obs} \label{discontinuidades de f-1 e intervalos invariantes}
The function $f \in IET(\S^{1},\mathcal{D})$ and $BP_{0}(f) \subset B$. 

For any $j \in \{0,...,M-1\}$ we have $$f\left(\bigcup_{h < k_{m_{j}}} B_{\tau^{h}}(m_{j})\right) = f\left(\bigcup_{i \in O(m_{j})} B_{i}\right) = \mathcal{I}_{j}((\S^{1})_{\lambda{j}}).$$

Then the set $\mathcal{I}_{j}((\S^{1})_{\lambda{j}})$ is invariant by the functions $f \circ \rho(t_{q}) \circ f^{-1}$ and $f \circ \rho(-t_{q}) \circ f^{-1}$ for all $q$.
\end{obs}

To prove that the function $f \circ \rho(t_{q}) \circ f^{-1}$ is continuous we will build a sequence of inclusion for the set \(BP_{0}(f \circ \rho(t_{q}) \circ f^{-1})\).

\begin{lema}
It holds $BP_{0}(f \circ \rho(t_{q}) \circ f^{-1}) \subset f(BP_{0}(\rho(t_{q})))$.
%The set $BP_{0}(h^{-1} \circ \rho(t_{q}) \circ h)$ verifies the inclusion $BP_{0}(h^{-1} \circ \rho(t_{q}) \circ h) \subset h^{-1}(BP_{0}(\rho(t_{q})))$.
\end{lema}
\begin{proof}

Applying Remark \ref{Discontinuidad de la inversa} we have the following inclusions 
$$BP_{0}(f \circ \rho(t_{q}) \circ f^{-1}) \subset BP_{0}(f^{-1}) \cup f(BP_{0}(f \circ \rho(t_{q})))  $$ $$\subset f(BP_{0}(f)) \cup f(BP_{0}(f \circ \rho(t_{q}))).$$
Also, $BP_{0}(f \circ \rho(t_{q})) \subset BP_{0}(\rho(t_{q})) \cup \rho(-t_{q})(BP_{0}(f))$. 

Since $BP_{0}(f) \subset B \subset BP_{0}(\rho(-t_q))$ and we have 
$$BP_{0}(f \circ \rho(t_{q})) \subset  BP_{0}(\rho(t_{q})) \cup \rho(-t_{q})(BP_{0}(f)) \subset BP_{0}(\rho(t_{q})) \cup \rho(-t_{q})(B) \subset $$ $$ \subset BP_{0}(\rho(t_{q})) \cup \rho(-t_{q})(BP_{0}(\rho(-t_{q}))) = BP_{0}(\rho(t_{q})).$$

On the other hand \(BP_{0}(f) \subset B \subset BP_{0}(\rho(t_{q}))\).

Then $BP_{0}(f \circ \rho(t_{q}) \circ f^{-1}) \subset f(BP_{0}(\rho(t_{q})))$.
\end{proof}

\begin{lema}
It holds $BP_{0}(f \circ \rho(t_{q}) \circ f^{-1}) \subset f(B)$.
\end{lema}

\begin{proof}
By the previous lemma to study the discontinuity of $f \circ \rho(t_{q}) \circ f^{-1}$ it is enough to consider the set $f(BP_{0}(\rho(t_{q}))$.

Let $x_{0} = f(a)$ for $a \in BP_{0}(\rho(t_{q})) \setminus B$ with \(x_{0} \in \S_{\lambda_{j}}\).

We separate into 4 sub cases.

Case 1: $a \in \left(B_{\tau^{h}(m_{j})}\right)_{+}$ and $h < k_{m_{j}} - 1$.

Let us calculate the left-limit in $x_{0}$.

Since $a \notin BP_{0}(f)$ (or equivalently $x_{0} = f(a) \in Cont(f^{-1})$) then $\lim_{x \to x_{0}^{-}} f \circ \rho(t_{q}) \circ f^{-1}(x) = \lim_{y \to a^{-}} f \circ \rho(t_{q})(y)$.

Since $a \in BP_{0}(\rho(t_{q})) \cap \left(B_{\tau^{h}(m_{j})}\right)_{+}$, then $\lim_{y \to a^{-}} \rho(t_{q})(y) = (b_{\tau^{h}(m_{j})+1})$, also there exists \(\delta > 0\) such that \(\rho(t_ {q})(B^*(a,\delta)_{-}) \subset (B_{\tau^{h}(m_{j})})_{+}\).

So by definition of $f$ we conclude 
$$\lim_{x \to x_{0}^{-}} f \circ \rho(t_{q}) \circ f^{-1}(x) = \lim_{z \to \left(b_{\tau^{h}(m_{j})+1}\right)^{-}} f(z) = $$
 $$ \lim_{z \to \left(b_{\tau^{h}(m_{j})+1}\right)^{-}} \mathcal{I}_{j}\left(\pi_{\lambda_{j}}\left(diam([b_{\tau^{h}(m_{j})},z)) +  \sum_{l < h}\mu(B_{\tau^{l}(m_{j})})\right)\right) = $$ 
$$=\mathcal{I}_{j}\left(\pi_{\lambda_{j}}\left(\sum_{i \leq h} \mu(B_{\tau^{l}(m_{j})})\right)\right).$$

In the other hand, since $a \in (BP_{0}(\rho(t_{q})) \setminus B) \cap \left(B_{\tau^{h}(m_{j})}\right)_{+}$ then $\rho(t_{q})(a) = b_{\tau^{h+1}(m_{j})}$, so $f \circ \rho(t_{q}) \circ f^{-1}(x_{0}) = f(b_{\tau^{h+1}(m_{j})}) = \mathcal{I}_{j}\left(\pi_{\lambda_{j}}\left(\sum_{i \leq h} \mu(B_{\tau^{l}(m_{j})})\right)\right)$.

Case 2: $a \in \left(B_{\tau^{k_{m_{j} - 1}}(m_{j})}\right)_{+}$.

Since $a \in BP_{0}(\rho(t_{q})) \setminus B$ then $\tau^{k_{m_{j}}}(m_{j}) = m_{j}$, so $\rho(t_{q})(a) = b_{m_{j}}$. 
Then we conclude that $f \circ \rho(t_{q}) \circ f^{-1}(x_{0}) = \mathcal{I}_{j}(\pi_{\lambda_{j}}(\lambda_{j}))$.

Applying the analogous argument to the case 1 we have $$\lim_{x \to x_{0}^{-}} f \circ \rho(t_{q}) \circ f^{-1}(x) = \lim_{z \to \left(b_{\tau^{k_{m_{j}}-1}(m_{j})+1}\right)^{-}} f(z) = $$
 $$ \lim_{z \to \left(b_{\tau^{k_{m_{j}}-1}(m_{j})+1}\right)^{-}} \mathcal{I}_{j}\left(\pi_{\lambda_{j}}\left(diam([b_{\tau^{k_{m_{j}}-1}(m_{j})},z)) +  \sum_{l < k_{m_{j}}-1}\mu(B_{\tau^{l}(m_{j})})\right)\right) = $$
$$ = \mathcal{I}_{j}\left(\pi_{\lambda_{j}}\left(\sum_{i \leq k_{m_{j}}-1} \mu(B_{\tau^{l}(m_{j})})\right)\right) = \mathcal{I}_{j}\left(\pi_{\lambda_{j}}(\lambda_{j})\right).$$

Case 3:  $a \in \left(B_{\tau^{h}(m_{j})}\right)_{-}$ and $h > 0$.

This case is similar to the case 1, in particular, $\lim_{x \to x_{0}^{-}} f \circ \rho(t_{q}) \circ f^{-1}(x) = \lim_{y \to a^{-}} f \circ \rho(t_{q})(y)$.

Since $a \in BP_{0}(\rho(t_{q})) \cap \left(B_{\tau^{h}(m_{j})}\right)_{-}$, then $\lim_{y \to a^{-}} \rho(t_{q})(y) = (b_{\tau^{h-1}(m_{j})+1})$, also there exists \(\delta > 0\) such that \(\rho(t_ {q})(B^*(a,\delta)_{-}) \subset (B_{\tau^{h-1}(m_{j})})_{+}\). By definition of $f$ we have

 $$\lim_{x \to x_{0}^{-}} f \circ \rho(t_{q}) \circ f^{-1}(x) = \lim_{z \to \left(b_{\tau^{h-1}(m_{j})+1}\right)^{-}} f(z) =$$ 
 $$= \lim_{z \to \left(b_{\tau^{h-1}(m_{j}) + 1}\right)^{-}} \mathcal{I}_{j}\left(\pi_{\lambda_{i}}\left(diam([b_{\tau^{h-1}(m_{j})},z)) +  \sum_{l < h-1}\mu(B_{\tau^{l}(m_{j})})\right)\right) = $$
 $$= \mathcal{I}_{j}\left(\pi_{\lambda_{i}}\left(\sum_{l < h} \mu(B_{\tau^{l}(m_{j})})\right)\right).$$

In the other hand 

$$f \circ \rho(t_{q}) \circ f^{-1}(x_{0}) = f \circ \rho(t_{q})(a) = f(b_{\tau^{h}(m_{j})}) = \mathcal{I}_{j}\left(\pi_{\lambda_{i}}\left(\sum_{l < h} \mu(B_{\tau^{l}(m_{j})})\right)\right).$$

Case 4: $a \in \left(B_{m_{j}}\right)_{-}$.

This case is analogous to the case 2.

Since $a \in BP_{0}(\rho(t))$ then $\sigma(m_{j}) = \tau^{k_{j}-1}(m_{j})$. 
Applying the analogous argument to the case 2 we have that
$$\lim_{x \to x_{0}^{-}} f \circ \rho(t_{q}) \circ f^{-1}(x) = $$ 
$$ = \lim_{z \to \left(b_{\tau^{k_{m_{j}}-1}(m_{j})+1}\right)^{-}} \mathcal{I}_{j}\left(\pi_{\lambda_{j}}\left(diam([b_{\tau^{k_{m_{j}}-1}(m_{j})},z)) +  \sum_{l < k_{m_{j}}}\mu(B_{\tau^{l}(m_{j})})\right)\right)$$ 
$$= \mathcal{I}_{j}(\pi_{\lambda_{j}}(\lambda_{j})).$$

In the other hand $f \circ \rho(t_{q}) \circ f^{-1}(x_{0}) = f \circ \rho(t_{q})(a) = f(b_{m_{j}}) = \mathcal{I}_{j}(\pi_{\lambda_{j}}(0))$.

We conclude that $BP_{0}(f \circ \rho(t_{q}) \circ f^{-1}) \subset f(B)$.
\end{proof}

\begin{lema}
The function $f \circ \rho(t_{q}) \circ f^{-1}$ is continuous in $\mathcal{I}_{j}((\S^{1})_{\lambda_{j}})$.
\end{lema}
\begin{proof}

For simplicity we denote by \(f \circ \rho(t_{q}) \circ f^{-1}\) the restriction \((f \circ \rho(t_{q}) \circ f^{-1})\vert_{\mathcal{I}_{j}(\S_{\lambda_{j}})}\).

We will study the set \(f(B \setminus \{b_{m_{j}}\})\) first.

Let $x_{0} = f(b_{\tau^{h}(m_{j})})$, with $0 < h < k_{m_{j}}$.

By definition of $f$ we have $x_{0} = \mathcal{I}_{j}\left(\pi_{\lambda_{j}}\left(\sum_{l < h} \mu(B_{\tau^{l}(m_{j})})\right)\right)$. The left-limit of $f^{-1}$ verifies that \\ $\lim_{x \to x_{0}^{-}} f^{-1}(x) = b_{\tau^{h-1}(m_{j}) + 1}$ and there exists \(\delta \in \R^{+}\) such that $ f^{-1}(x) \in (B_{\tau^{h-1}(m_{j})})_{+}$, for any \(x \in B^{*}(x_{0},\delta)_{-}\).

Since $0< h < k_{m_{j}}$ (i.e. $\tau^{h}(m_{j}) \neq \tau^{h-1}(m_{j}) = \sigma(\tau^{h}(m_{j})) $) we have that $s = -t_{q}$ or $s = t_{q}$ verifies $\rho(s)(b_{\tau^{h}(m_{j})}) \in (B_{\tau^{h-1}(m_{j})})$, in particular, exists \(c \in B_{\tau^{h}(m_{j})}\) such that $[b_{\tau^{h}(m_{j})},c)$ is an interval of type 1, then $\rho(s)(b_{\tau^{h}(m_{j})}) \in (B_{\tau^{h-1}(m_{j})})_{+}$.
Let us see that both cases.

%Since $h > 0$ (i.e. $\tau^{h}(m_{j}) \neq \tau^{h-1}(m_{j}) = \sigma(\tau^{h}(m_{j})) $) we have that $\rho(t_{q})(b_{\tau^{h}(m_{j})}) \in (B_{\tau^{h}(m_{j})})_{-}$ and $\rho(-t_{q})(b_{\tau^{h}(m_{j})}) \in (B_{\tau^{h-1}(m_{j})})_{+}$ or $\rho(t_{q})(b_{\tau^{h}(m_{j})}) \in B({\tau^{h-1}(m_{j})})_{+} $ and $\rho(-t_{q})(b_{\tau^{h}(m_{j})}) \in (B_{\tau^{h}(m_{j})})_{-}$. Let us see that cases:

%Case 2.1: $\rho(-t_{q})(b_{\tau^{h}(m_{j})}) \in (B_{\tau^{h-1}(m_{j})})_{+}$.
%
%Applying Proposition \ref{t o -t}, \(\rho(t_{q})\) has not interval of type 1 in \((B_{\tau^{h}(m_{j})})_{-}\) then $\rho(t_{q})(b_{\tau^{h}(m_{j})}) \in (B_{\tau^{h}(m_{j})})_{-}$.
%
%By composition with $\rho(t_{q})$ we have $$\lim_{x \to x_{0}^{-}} \rho(t_{q}) \circ f^{-1}(x) = \lim_{y \to \left(b_{\tau^{h-1}(m_{j}) + 1}\right)^{-}} \rho(t_{q})(y) = c \in (B_{\tau^{h}(m_{j})})_{-} \setminus B.$$
%
%Since $c \in BP_{0}(\rho(-t_{q})) \setminus B$ and $c \in (B_{\tau^{h}(m_{j})})_{-}$ there exists $d \in (B_{\tau^{h}(m_{j})})_{+}$ such that $[c,d)$ is an interval of type 2 for $\rho(-t_{q})$, then $\rho(-t_{q})(c) = b_{\tau^{h}(m_{j})}$, equivalently $\rho(t_{q})(b_{\tau^{h}(m_{j})}) = c$.
%
%In summary $\lim_{x \to x_{0}^{-}} \rho(t_{q}) \circ f^{-1}(x) = \rho(t_{q}) \circ f(x_{0}) = c \in Cont(f)$, then $f \circ \rho(t_{q}) \circ f^{-1}$ is continuous in $x_{0}$.

Case 1:  $\rho(-t_{q})(b_{\tau^{h}(m_{j})}) \in (B_{\tau^{h-1}(m_{j})})_{+}$

There exist \(c,d \in B_{\tau^{h}(m_{j})}\) such that \([b_{\tau^{h}(m_{j})},c)\) is an interval of type 1 of \(\rho(-t_{q})\) and \([c,d)\) is an interval of type 2 of \(\rho(-t_{q})\).

Since \(c \notin B\) and  \(\rho(-t_{q})([c,d)) \subset B_{\tau^{h}(m_{j})}\) we have \(\rho(-t_{q})(c) = b_{\tau^{h}(m_{j})}\).

Since \(\rho(-t_{q})([b_{\tau^{h}(m_{j})},c)) \subset (B_{\tau^{h-1}(m_{j})})_{+}\) is an interval of type 1 of \(\rho(t_{q})\), we have that \\ \(\rho(-t_{q})([b_{\tau^{h}(m_{j})},c)) = [\rho(-t_{q})(b_{\tau^{h}(m_{j})}),b_{\tau^{h-1}(m_{j})+1})\).
In particular \(\lim_{x \to (b_{\tau^{h-1}(m_{j})+1})^{-}} \rho(t_{q})(x) = c\).

In summary:
$$\lim_{x \to x_{0}^{-}} f \circ \rho(t_{q}) \circ f^{-1}(x) = \lim_{x \to (b_{\tau^{h-1}(m_{j})+1})^{-}} f \circ \rho(t_{q})(x) = \lim_{x \to c^{-}} f(x) = $$ 
$$f(c) = f \circ \rho(t_{q})(b_{\tau^{h}(m_{j})}) =  f \circ \rho(t_{q}) \circ f^{-1}(x_{0}),$$
then \(f \circ \rho(t_{q}) \circ f^{-1}\) is continuous in \(x_{0}\).

Case 2: $\rho(t_{q})(b_{\tau^{h}(m_{j})}) \in B(_{\tau^{h-1}(m_{j})})_{+}$.

Adapting the previous argument we have that there \(c \in 
B_{\tau^{h}(m_{j})}\) such that \([b_{\tau^{h}(m_{j})+1},c)\) is an interval of type 1 of \(\rho(t_{q})\). Also we have that \(\rho(t_{q})([b_{\tau^{h}(m_{j})}),c) = [\rho(t_{q})(b_{\tau^{h}(m_{j})}), b_{\tau^{h-1}(m_{j})+1})\subset (B_{\tau^{h-1}(m_{j})})_{+}\) is an interval of type 1 of \(\rho(-t_{q})\). 

Let \(d = \rho(t_{q})(b_{\tau^{h}(m_{j})})\), there exists \(d^{\prime} \in B_{\tau^{h-1}(m_{j})}\) such that \([d^{\prime},d)\) is an interval of type 2 of \(\rho(-t_{q})\), in particular \(\lim_{x \to d^{-}} \rho(-t_{q})(x) = b_{\tau^{h-1}(m_{j})+1}\), equivalently \(\lim_{x \to (b_{\tau^{h-1}(m_{j})+1})^{-}} \rho(t_{q})(x) = d = \rho(t_{q})(b_{\tau^{h}(m_{j})})\).

In summary:
$$\lim_{x \to x_{0}^{-}} f \circ \rho(t_{q}) \circ f^{-1}(x) = \lim_{x \to (b_{\tau^{h-1}(m_{j})+1})^{-}} f \circ \rho(t_{q})(x) = \lim_{x \to d^{-}} f(x) = $$ 
$$f(d) = f \circ \rho(t_{q})(b_{\tau^{h}(m_{j})}) =  f \circ \rho(t_{q}) \circ f^{-1}(x_{0})$$
then \(f \circ \rho(t_{q}) \circ f^{-1}\) is continuous in \(x_{0}\).
%
%Case 2.2:  $\rho(t_{q})(b_{\tau^{h}(m_{j})}) \in B(_{\tau^{h-1}(m_{j})})_{+} $ and $\rho(-t_{q})(b_{\tau^{h}(m_{j})}) \in (B_{\tau^{h}(m_{j})})_{-}$.
%
%Let $c = \rho(t_{q}) \circ f^{-1}(x_{0}) \in (B_{\tau^{h-1}(m_{j})})_{+}$.
%Since \(c \in BP_{0}(\rho(-t_{q}))\) there exists $d \in B_{\tau^{h-1}(m_{j})}$ such that $[d,c)$ is an interval of type 2 for $\rho(-t_{q})$.
%Then $\lim_{y \to c^{-}} \rho(-t_{q})(y) = b_{\tau^{h-1}(m_{j})+1}$ or equivalently $\lim_{y \to (b_{\tau^{h-1}(m_{j})+1})^{-}} \rho(t_{q})(y) = c$.
%
%In summary $\lim_{x \to x_{0}^{-}} \rho(t_{q}) \circ f^{-1}(x) = \rho(t_{q}) \circ f^{-1}(x_{0}) = c \in Cont(f)$, then $f \circ \rho(t_{q}) \circ f^{-1}$ is continuous in $x_{0}$.

We conclude that \(BP_{0}(f \circ \rho(t_{q}) \circ f^{-1}) \subset \{ f(b_ {m_{j}})\}\), in particular \(\# BP_{0}( f \circ \rho(t_{q})) \circ f^{-1}) \leq 1\) then \(f \circ \rho(t_{q}) \circ f^{-1}\) is continuous.

%Finally, let's prove $b_{m_{j}} \notin BP_{0}(f^{-1} \circ \rho(t_{q})) \circ f)$.
%
%If $\sigma(m_{j}) = m_{j}$ then $\tau (\tau^{k_{m_{j}}-1}(m_{j})) =  \tau^{k_{m_{j}}-1}(m_{j})$, in particular $\rho(t_{q})(b_{m_{j}}) = b_{m_{j}}$ and $\lim_{y \to b_{\tau^{k_{m_{j}}-1}(m_{j})+1}} \rho(t_{q})(y) = b_{\tau^{k_{m_{j}}-1}(m_{j})+1}$ . In this condition $f^{-1} \circ \rho(t_{q}) \circ f(c_{j}) = c_{j}$ and $\lim_{x \to c_{j+1}^{-}} f \circ \rho(t_{q}) \circ f(x) = c_{j+1}$, in particular verifies the condition ii.
%
%In the other case, is enough adapt the argument of the case 2 to this index (i.e. $h = 0$).
\end{proof}

We need to extend this result from the sequence \(t_{q}\) to any \(t \in \R\).

\begin{lema} \label{Ij is inveiant for rho(t)}
The set $\mathcal{I}_{j}((\S^{1})_{\lambda{j}})$ is invariant by the functions $f \circ \rho(t) \circ f^{-1}$ for all $t \in \R$.
\end{lema}

\begin{proof}
Suppose that there exists \(t \in \R^{+}\) and \(x \in  (\S^{1})_{\lambda{j}}\) such that \(f \circ \rho(t) \circ f^{-1}(x) \notin \mathcal{I}_{j}((\S^{1})_{\lambda{j}})\), the case \(t \in \R^{-}\) is analogous. 
This property is equivalent to the next one: there exist \(h, h^{\prime}, j^{\prime} \in \N\) and \(a \in \S^{1}\) such that \(j^{\prime} \neq j\), \(h < k_{j}\), \(h^{\prime} < k_{j^{\prime}}\) \(a \in B_{\tau^{h}(m_{j})}\) and \(\rho(t)(a) \in  B_{\tau^{h^{\prime}}(m_{j^{\prime}})}\).

Applying the right continuity of \(\rho(t)\) there exists \(c \in B_{\tau^{h}(m_{j})} \setminus \{a\}\) such that \(\rho(t)([a,c)) \subset B_{\tau^{h^{\prime}}(m_{j^{\prime}})}\).

Since \(\rho\) is continuous there exists \(\delta^{\prime} > 0\) such that for any \(s \in (t-\delta^{\prime},t+\delta^{\prime})\) there exists \(a_{s} \in [a,c)\) such that \(\rho(s)(a_{s}) \in  B_{\tau^{h^{\prime}}(m_{j^{\prime}})}\).

There exists a finite sequence \(n_{k}\) such that \(\sum_{k \leq N} t_{q_{n_{k}}} \in (t-\delta,t+\delta)\). So we can conclude that  $\mathcal{I}_{j}((\S^{1})_{\lambda{j}})$ is not invariant for \(\rho(\sum_{k \leq N} t_{q_{n_{k}}}) = \rho(t_{q_{n_{N}}}) \circ ... \circ \rho(t_{q_{n_{0}}})\) which is a contradiction.
\end{proof}

Since \(\int_{\mathcal{I}_{j}((\S^{1})_{\lambda_{j}})} \d_{\lambda_{j}}(g_{1}(x),g_{2}(x)) + \d_{\lambda_{j}}(g_{1}^{-1}(x),g_{2}^{-1}(x)) \leq d_{D}(g_{1},g_{2})\) by Remark \ref{conjugar es continuo} the homomorphism \(\tilde{\rho} = (f \circ \rho \circ f^{-1})\vert_{\mathcal{I}_{j}}\) is continuous. 
Since \(\tilde{\rho}(t_{q}) \in \mathcal{AC}_{+}(\mathcal{I}_{j}(\S^{1})_{\lambda_{j}})\), by Proposition \ref{continuidad en una sucesion implica continuidad} we have that \(\tilde{\rho}(\R) \subset \mathcal{AC}_{+}(\mathcal{I}_{j}(\S^{1})_{\lambda_{j}})\), then we finish the proof of the main theorem.

\section{Considerations of the main theorem}

In the case $\rho(\R) \subset IET(\S^{1})$ (similar to the hypothesis of main theorem of \cite{Novak1}) we have $\# [\mathring{B_{i}} \cap BP_{0}(\rho(t))] \leq 1$ for any $t$ and $i$. But this is not true in general case.

\begin{ej} \label{ejemplo con dos discontinuidades en un intervalo}
Let $\Phi:(\S^{1})_{\frac{1}{2}} \times \R \to (\S^{1})_{\frac{1}{2}}$ a differential flow with two fixed points, $\frac{1}{8}$ and $\frac{3}{8}$. The first one is a repeller and the other is  an attractor . Note that \(\Phi(t)(\frac{1}{4}) \in [0,\frac{1}{4}) \subset (\S^{1})_{1/2}\) and \(\lim_{x \to (\frac{1}{2})^{-}} \Phi(t)(x) \in [0,\frac{1}{4}) \subset (\S^{1})_{1/2}\).

Let $D = (\S^{1})_{\frac{1}{2}} \uplus (\S^{1})_{\frac{3}{8}} \uplus (\S^{1})_{\frac{1}{8}}$.

We denote by \(\psi:\R \to C^{1}(D)\) the homomorphism defined by \(\psi(t)(x) = \Phi(t)(x)\) if \(x \in (\S^{1})_{\frac{1}{2}}\) and \(\psi(t)(x) = x\) if \(x \notin (\S^{1})_{\frac{1}{2}}\).

Let \(h \in IET(D,\S^{1})\) defined by 

$$h(x) =  \left \{ \begin{matrix}
x & \text{ if } x \in (\S^{1})_{1/2} \text{ and } x < \frac{1}{4} \\
x + \frac{1}{4} &  \text{ if } x \in (\S^{1})_{\frac{3}{8}}  \\
x + \frac{3}{8} & \text{ if } x \in (\S^{1})_{\frac{1}{2}} \text{ and } x \geq \frac{1}{4} \\
x + \frac{7}{8} & \text{ if } x \in (\S^{1})_{\frac{1}{8}} 
\end{matrix}\right.$$

Let $\rho:\R \to \ACI$ the homomorphism defined by \(\rho(t) = h \circ \Phi(t) \circ h^{-1}\)

It is easy to check that $B = \{0,\frac{1}{4},\frac{5}{8},\frac{7}{8}\}$.

Note that for \(x \in [0,\frac{1}{4}) \subset (\S^{1})_{\frac{1}{2}}\) we have that \(h(x) \in [0,\frac{1}{4}) \subset \S^{1}\). Then \(B_{1} = [0,\frac{1}{4})\) has two intervals of type 1 for \(\rho(t)\) for \(t > 0\), in particular $\#(\mathring{B_{1}} \cap BP_{0}(\rho(t)) ) = 2$.

\end{ej}

The condition of group homomorphism is necessary.

\begin{ej}
Let $\Phi_{n}:\R\to \ACI$ be the function defined by

 $$\Phi_{n}(t)(x) = \left \{ \begin{matrix} x + \frac{t - \lfloor t \rfloor}{n} & \mbox{ if } x \in \left[\frac{k}{n}, \frac{k+1}{n} -  \frac{t - \lfloor t \rfloor}{n}\right) 
\\ x + \frac{t - \lfloor t \rfloor}{n} - \frac{1}{n}  & \mbox{ if } x \in \left[\frac{k+1}{n} -  \frac{t - \lfloor t \rfloor}{n} , \frac{k+1}{n}\right)  \end{matrix}\right.$$

The functions $\Phi_{n}$ are continuous and $\Phi_{n}(m) = id$ for any $m \in \Z$.

The function $\Phi:\R\to \ACI$ defined by $\Phi(t)(x) = \Phi_{\lfloor \vert t \vert\rfloor}(t)(x)$ are continuous but $\sharp \circ \Phi$ is not bounded.
\end{ej}

The proof of the Proposition \ref{sharp finite close to 0} is valid in complete groups with $0$ as limit point, i.e. $\sharp$ is bounded in a neighbourhood of $e$. But is not valid in not complete groups.

\begin{prop}
Let $G$ be a group and $d^{\prime}$ a distance such that $(G,d^{\prime})$ is a topological group and a complete metric space, such that $e \in G$ is a limit point.

For any $\tau: G \to \ACI$ continuous homomorphism there exists $U$ a neighbourhood of $e$ such that $\sharp \circ \tau\vert_{U}$ is finite.
\end{prop}

\begin{ej}
Let $G = \{\frac{k}{2^{n}} : k,n \in \Z \} \subset \Q$ and $d_{\Q}$ the canonical distance in $\Q$. Note that $(G,d_{\Q})$ is not complete and any $p \in G$ is a limit point.

The subset $G_{0} = \{ \frac{1}{2^{n}} : n \in \N \}$  is a generator of $G$, so any group morphism $\rho:G  \to \ACI$ is determined by the restriction in \(G_{0}\).

We define $\rho:G_{0} \to \ACI$ by 

$\rho(\frac{1}{2^{n}})(x) = \left \{ \begin{matrix} x & \mbox{ if } x \in [0,\frac{1}{2^{n}}), \\
x + \frac{1}{2^{n+1}}  & \mbox{ if } x \in \left[\frac{1}{2^{m}} ,\frac{1}{2^{m-1}}- \frac{1}{2^{n+1}}\right) \mbox{ with } m \leq n,
 \\ x - (\frac{1}{2^{m-1}} -  \frac{1}{2^{n+1}}) & \mbox{ if } x \in \left[\frac{1}{2^{m-1}}- \frac{1}{2^{n+1}}, \frac{1}{2^{m-1}}\right) \mbox{ with } m \leq n.
 \end{matrix}\right.$

\begin{figure}[h]
\includegraphics[scale=0.78]{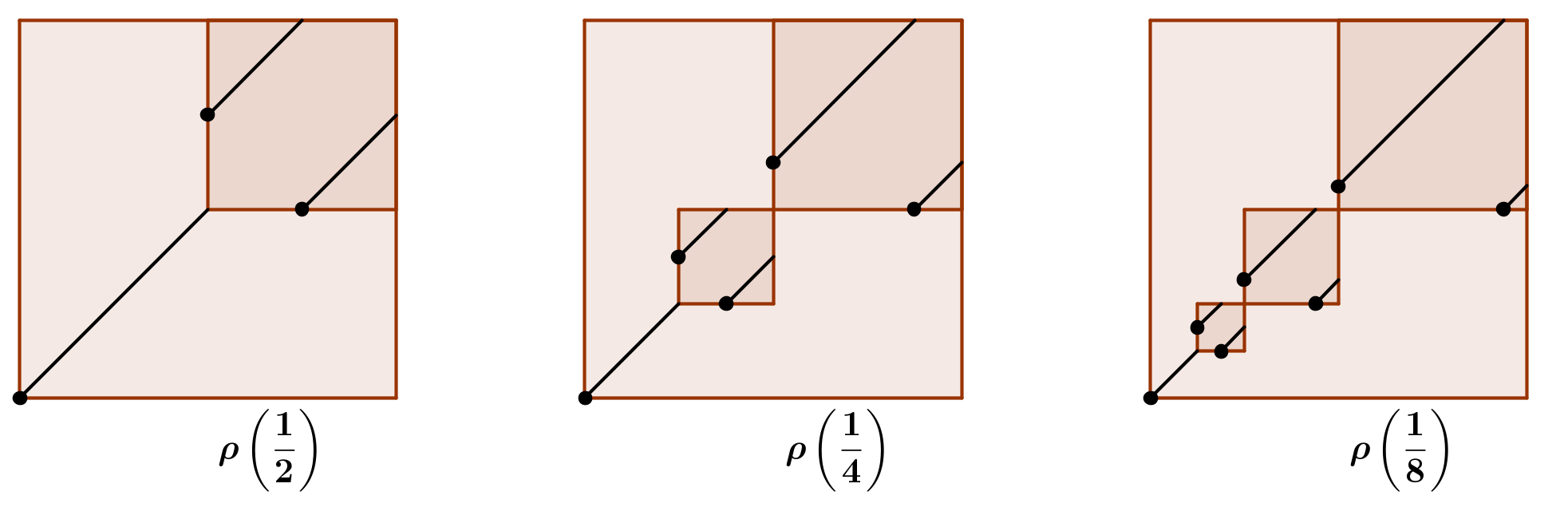}
\end{figure}

Note that $\rho(\frac{1}{2^{n}}) =  \rho(\frac{1}{2^{n+1}}) \circ \rho(\frac{1}{2^{n+1}})$, so $\rho(G_{0}) \subset \ACI$ is a commutative subset. 

Let $\frac{m_{1}}{2^{n_{1}}}$ and $\frac{m_{2}}{2^{n_{2}}}$ two representation of the same element of $G$ with $n_{1} \leq n_{2}$. Then $m_{2} = m_{1} 2^{n_{2} - n_{1}}$ and $(\rho(\frac{1}{2^{n_{2}}}))^{2^{n_{2} - n_{1}}} = \rho(\frac{1}{2^{n_{1}}})$, in particular $\rho(\frac{1}{2^{n_{2}}})^{m_{2}} = \rho(\frac{1}{2^{n_{1}}})^{m_{1}}$.

Finally we can define $\rho:G \to \ACI$ by $\rho(\frac{m}{2^{n}}) = (\rho(\frac{1}{2^{n}}))^{m}$.

Note that $$\rho\left(\frac{m_{1}}{2^{n_{1}}} + \frac{m_{2}}{2^{n_{2}}}\right) = \rho\left(\frac{m_{1} 2^{n_{2}}+ m_{2}2^{n_{1}}}{2^{n_{1}+n_{2}}}\right) = \left(\rho\left(\frac{1}{2^{n_{1}+n_{2}}}\right)\right)^{m_{1}2^{n_{2}} + m_{2}2^{n_{1}}}$$

$$ = \left(\rho\left( \frac{1}{2^{n_{1}+ n_{2}}} \right) \right)^{m_{1}2^{n_{2}}} \circ \left(\rho\left(  \frac{1}{2^{n_{1}+ n_{2}}} \right) \right)^{m_{2}2^{n_{1}}} = \rho\left(\frac{m_{1}}{2^{n_{1}}} \right) \circ \rho\left(\frac{m_{2}}{2^{n_{2}}}\right)$$
then $\rho$ is a group homomorphism.

Given $n, m \in \N$ with $m\leq n$ we have that 
$$\int_{\frac{1}{2^{m}}}^{\frac{1}{2^{m-1}}} \d\left(id(x),\rho\left(\frac{1}{2^{n}}\right)(x)\right) = \frac{1}{2^{n+1}}\left(\frac{1}{2^{m}}- \frac{1}{2^{n+1}}\right)
+
\left(\frac{1}{2^{m-1}}- \frac{1}{2^{n+1}}\right)\frac{1}{2^{n+1}} $$ 
$$\leq \frac{1}{2^{m-1}}\frac{1}{2^{n+1}}.$$ 

Then $d(id,\rho(\frac{1}{2^{n}})) \leq \frac{1}{2^{n-1}}$, so we  conclude that \(\rho\) is continuous.

Finally, since $\sharp \rho(\frac{1}{2^{n}}) = 2n$, we have that $\sharp \circ \rho$ is not finite in any neighbourhood of 0.
\end{ej}

The previous example can be adapted to a non continuous homomorphism of $\R$ to $\ACI$.

\begin{ej}
Let $\mathcal{B} = \{v_{i}, i \in I\}$ a Hamel basis of $\R$ over $\Q$, with $I$ a total well-ordered set.

Let $\rho$ the morphism of the previous example.

Let $\rho_{R}:\R \to \ACI$ the morphism define as follow. 
For $v_{i}$ in the basis $\mathcal{B}$ and \(\lambda_{i} \in \Q^{+}\) by \(\rho_{R}(\lambda_{i} v_{i}) = id\) if \(\#\{j \in I : j \leq i\}\) is not finite, in other case let \(n = \#\{j \in I : j \leq i\}\) and

$\rho_{R}(\lambda_{i}v_{i})(x) = \left \{ \begin{matrix} x & \mbox{ if } x \in [0,\frac{1}{2^{n}}), \\
x + \lambda_{i}\left(\frac{1}{2^{n+1}}\right)  & \mbox{ if } x \in \left[\frac{1}{2^{m}} ,\frac{1}{2^{m-1}}- \lambda_{i}\frac{1}{2^{n+1}}\right) \mbox{ with } m < n,
 \\ x -\left(\frac{1}{2^{m-1}} - \lambda_{i}\frac{1}{2^{n+1}}\right) & \mbox{ if } x \in \left[\frac{1}{2^{m-1}}- \lambda_{i}\frac{1}{2^{n+1}}, \frac{1}{2^{m-1}}\right) \mbox{ with } m < n.
 \end{matrix}\right. .$

For a  generic vector \(v = \sum_{i \in I_{0}} \lambda_{i} v_{i}\) we define \(\rho_{R}(v) = \circ_{i \in I_{0}} \rho_{R}(\lambda_{i}v_{i})\)
 
This homomorphism verifies that \(\sharp \circ \rho_{R}\) is not bounded. 
\end{ej}

\bibliographystyle{plain} % Choose Phys. Rev. style for bibliography
\bibliography{sample}

\end{document}